\documentclass{amsart}

\usepackage{va, stycb}

\usepackage[pdftex]{hyperref}
\hypersetup{colorlinks,citecolor=blue,linktocpage,hyperindex=true,backref=true}

\usepackage{tikz-cd}

\usepackage{float}
\usepackage{placeins}

\usepackage{enumitem}
\usepackage{bm}
\usepackage{verbatim}
\usepackage[singlelinecheck=false]{caption}

\usepackage{todonotes}

\title[Explicit KSBA compactifications for secondary Burniat surfaces]%
{Explicit KSBA compactifications of moduli spaces\\
  of secondary and tertiary Burniat surfaces}
\date{May 12, 2024}  

\author{Valery Alexeev} 
\email{valery@uga.edu}
\address{Department of Mathematics, University of Georgia, Athens GA 30602, USA}
\author{Xiaoyan Hu}
\email{xiaoyan.mga@gmail.com} 
\address{Unaffiliated} 

\begin{document}
\maketitle

\begin{abstract}
  We describe explicitly the geometric KSBA compactifications, obtained
  by adding slc surfaces~$X$ with ample canonical class,
  of moduli spaces of Burniat surfaces of degrees $K^2=5$,
  $4$ and $3$. 
\end{abstract}

\setcounter{tocdepth}{1}

\tableofcontents

\section{Introduction}
\label{sec:intro}

Burniat surfaces of degree $3\le d=K_X^2\le 6$ are surfaces of general
type which can be obtained as $\bZ_2^2$-covers of degree~$d$ del Pezzo
surfaces, branched over the preimages of some very special line
configurations in $\bP^2$. The primary Burniat surfaces are those of
degree $6$, they are $\bZ_2^2$-covers of the Cremona surface
$\Sigma= \Bl_3\bP^2$. The secondary (resp. tertiary) Burniat surfaces
have degrees $5$, $4$ (resp. $3$). They are obtained by making further
blowups of special configurations where some additional triples of
lines are incident. See Section~\ref{sec:burniat-surfaces} for precise
definitions.

There are four types of secondary and tertiary Burniat surfaces, two
generically smooth: in degree $5$ and in degree $4$ non-nodal type
(which we denote $4a$), and two generically nodal: in degree~$4$ nodal
type (denoted $4b$) and in degree $3$. In degree~$d$ the moduli space
$M\ubur_d$ has dimension $d-2$.  In the two smooth types it is an
irreducible component of the moduli space of surfaces of general
type. In the two nodal types it is an irreducible closed subset in a
component of higher dimension, see \cite{bauer2010burniat-II,
  bauer2014burniat-II-erratum, bauer2013burniat-III}.

In this paper for each case $d=5,4a,4b,3$ we explicitly describe the
KSBA compactification $\oM\ubur_d$ of $M\ubur_d$, which we define as
the normalization of the closure of $M\ubur_d$ in the KSBA space
$\oM\uslc$ of stable surfaces with slc singularities and ample
canonical class, taken with the reduced scheme structures. Our main
reference for the KSBA moduli space $\oM\uslc$ is
\cite{kollar2023families-of-varieties}.  We do this as an application
of the paper \cite{alexeev2024explicit-compactifications} in which the
first author and Rita Pardini solved the same problem for the primary
Burniat surfaces. An earlier version of this work, with less explicit
descriptions and done by a different method, formed a PhD thesis of
the second author \cite{hu14compactifications-of-moduli} directed by
the first author.

Our main results are as follows.  In these theorems, by the boundary
of the compactification we mean the closed subset parameterizing
surfaces with worse than canonical singularities. The origin of a
toric variety $V$ is the point $1\in (\bC^*)^r \subset V$.

\begin{theorem}\label{thm-intro:3}
  The KSBA compactification $\oM\ubur_{3}$ is an $S_2$-quotient of $\bP^1$
  (and is isomorphic to $\bP^1$).
  The boundary $\oM\ubur_3\setminus M\ubur_3$
  consists of two divisors.
  The underlying moduli stack is a $\bZ_2^2$-gerbe over the quotient
  stack $[\bP^1: (C_3\times S_2)]$.
\end{theorem}

\begin{theorem}\label{thm-intro:4a}
  In the non-nodal degree $4$ case, the KSBA compactification
  $\oM\ubur\foura$ is a $C_3\times S_2$-quotient of the smooth toric
  Cremona surface $\Sigma=\Bl_3\bP^2$.  The boundary
  $\oM\ubur\foura\setminus M\ubur\foura$ consists of one divisor and
  one point.  The underlying moduli stack is a $\bZ_2^2$-gerbe over
  the quotient stack $[\Sigma: (C_3\times S_2^2)]$.
\end{theorem}

\begin{theorem}\label{thm-intro:4b}
  In the nodal degree $4$ case, the KSBA compactification $\oM\ubur\fourb$ is an
  $S_2^2$-quotient of the surface $\Bl_5(\bP^1\times\bP^1)$, 
  the blowup of the toric surface $\bP^1\times\bP^1$ at the four fixed
  points and at the origin.
  The boundary
  $\oM\ubur\fourb\setminus M\ubur\fourb$ consists of five
  divisors. 
  The underlying moduli stack is a $\bZ_2^2$-gerbe over the quotient
  stack $[\Bl_5(\bP^1\times\bP^1) : S_2^2]$.
\end{theorem}

\begin{theorem}\label{thm-intro:5}
  The KSBA compactification $\oM\ubur_5$ is a $C_3\times S_2$-quotient
  of a smooth $3$-dimensional variety $Z_5$ which is the blowup at the
  origin of the smooth projective toric variety $Z_5\tor$ explicitly
  described in Section~\ref{sec:deg5}.  The boundary
  $\oM\ubur_5\setminus~M\ubur_5$ consists of seven
  divisors.
  The underlying moduli stack is a $\bZ_2^2$-gerbe over the quotient
  stack $[Z_5: (C_3\times S_2)]$.
\end{theorem}


The description of $\oM_d\ubur$ as stacks in these theorems gives both
the isomorphism classes and automorphism groups of the surfaces.

We work over $\bC$.  Most results can be
extended to any field of characteristic different from~$2$ since our
moduli spaces are obtained from explicit toric varieties by simple
blowups and taking quotients, and the abelian covers $X\to Y$ are
$\mu_2^2$-covers of surfaces. We do not belabor this point, as this is
not the focus of the present paper.

In \cite{alexeev2025kappa-classes} the results of this paper were
applied to investigate the enumerative geometry and the kappa classes
(generalized MMM classes)
of the moduli space $\oM\ubur\foura$.  We also note that the methods
of \cite{alexeev2024explicit-compactifications} and of the present paper
apply to other families of general type constructed from planar line
configurations. This includes for example Kulikov
surfaces studied in \cite{kulikov2004old-examples, chan2013kulikov-surfaces}. 

The version of this paper on arXiv includes a sage \cite{sagemath}
script with explicit computations of the fans, symmetry groups and
moduli spaces appearing here.

\begin{acknowledgements}
  The first author was partially supported by the NSF under
  DMS-2201222. We thank Rita Pardini for useful comments.
\end{acknowledgements}

\section{Definition of Burniat surfaces} 
\label{sec:burniat-surfaces}

Burniat surfaces can be defined in two ways, either as
$\bZ_2^2$-covers $\pi\colon X\to Y$ of del Pezzo surfaces as in
\cite{burniat1966sur-les-surfaces, peters1977on-certain-examples} or
using quotients of certain hypersurfaces in products of three elliptic
curves as in \cite{peters1977on-certain-examples,
  inoue1994some-new-surfaces}. We use the first way. By the general
theory of abelian covers \cite{pardini1991abelian-covers},
a Galois $\bZ_2^2$-cover of a smooth rational surface $Y$ can be
defined by three divisors $R$, $G$, $B$ (for which we use the three
primary colors red, green and blue) such that $R+G$, $G+B$ and $B+R$
are $2$-divisible in the class group of $Y$.
\begin{figure}[htbp]
  \centering
  \includegraphics[width=340pt]{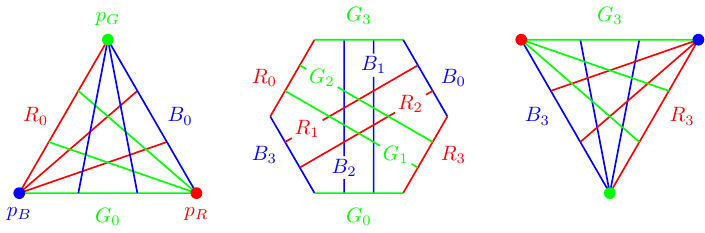}      
  \caption{Burniat arrangements on $\bP^2$ and $\Sigma=\Bl_3\bP^2$}
  \label{fig-burniat-config}
\end{figure}
For primary Burniat surfaces, i.e. those of degree~$6$, $Y=\Sigma$ is the
Cremona surface, and the three divisors are split into $12$ curves
$R_i$, $G_i$, $B_i$, $i=0,1,2,3$ as shown in the middle panel of 
Figure~\ref{fig-burniat-config}.  We call the curves with $i=0,3$
\emph{boundary} and those with $i=1,2$ \emph{interior}. For secondary
(degrees $5$ and $4$) and tertiary (degree $3$)
Burniat surfaces, one further considers the special configurations when
some of the interior curves pass through one, two, or three common
points, as in Figure~\ref{fig-secondary-burniat}.
\begin{figure}[htbp]
  \centering
  \includegraphics[width=360pt]{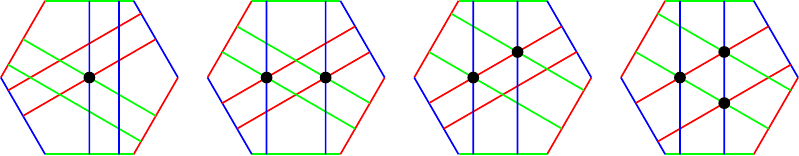}
  \caption{Cases of degree 5, 4 non-nodal, 4 nodal, 3}
  \label{fig-secondary-burniat}
\end{figure}

\begin{definition}
  We define secondary and tertiary Burniat surfaces as follows.   
  Consider the following diagram:
  \begin{center}
    \begin{tikzcd}
      X_6  \ar[d, "\pi_6"] &
      X' \ar[l, swap, "\wt g"] \ar[r, "\wt h"] \ar[d, "\pi'"] &
      X \ar[d, "\pi"] \\
      Y_6  &
      Y' \ar[l, swap, "g"] \ar[r, "h"] &
      Y
    \end{tikzcd}
  \end{center}
  Here, $Y_6=\Sigma$ is the Cremona surface and
  $\pi_6\colon X_6\to Y_6$ is the $\bZ_2^2$-cover defined by the
  divisors $R=\sum R_i$, $G=\sum G_i$, $B=\sum B_i$.

  Let $f\colon Y'\to Y_6$ be the blowup of
  $n=1$, $2$ or $3$ distinguished points of
  Figure~\ref{fig-secondary-burniat}, with the exceptional curves
  $E_1,\dotsc, E_n$, and 
  the morphism $\pi'\colon X'\to Y'$
  be the $\bZ_2^2$-cover defined by the strict preimages $R'_i$,
  $G'_i$, $B'_i$; the curves $E_j$ are not part of the ramification
  divisors.
  One easily computes that
  \begin{displaymath}
    K_{Y'} + \tfrac12 D' = -\tfrac12 K_{Y'} = g^*(K_{Y_6} + \tfrac12
    D_6) - \sum_{s=1}^n \tfrac12 E_s,
  \end{displaymath}
  where $D_6=R+G+B$ and $D' = R' + G' + B'$ are the total branch divisors on
  $Y_6$ and $Y'$, and $D'$ is the strict preimage of $D_6$.
  
  Let $C'$ be one of the interior curves on $Y'$. For its reduced preimage
  on $X'$ one has
  \begin{math}
    (\pi'{}\inv C')^2 = (C')^2 = -n(C'),
  \end{math}
  where $n(C')$ is the number of the blown up special points on $C'$.
  It easily follows that if all $n(C')\le1$, as in the first two cases
  of Figure~\ref{fig-secondary-burniat}, then $Y'$ is a del Pezzo
  surface. When some $n(C')=2$ as in the last two nodal cases, $Y'$ is
  an almost del Pezzo surface, with big and nef $-K_{Y'}$. If
  $n(C')=2$ then $C'$ and $\pi'{}\inv C'$ are both smooth rational
  $(-2)$-curves. Let $h$ and $\wt h$ be the morphisms contracting these
  $(-2)$-curves (if there are any) to the nodes,
  i.e. $A_1$-singularities on $Y$ and $X$ respectively. We have an
  induced $\bZ_2^2$-cover $\pi\colon X\to Y$.

  Denote the divisor
  $h(D')$ by $D$. Then one has
  \begin{displaymath}
    K_{X'} = \pi'{}^*(K_{Y'}+\tfrac12 D') = \tfrac12 \pi'{}^*(-K_{Y'}),
    \quad
    K_{X} = \pi{}^*(K_{Y}+\tfrac12 D) = \tfrac12 \pi{}^*(-K_{Y}),
  \end{displaymath}
  It follows that $X$ is a surface with $0$, $1$ or $3$ nodes and with
  an ample canonical class of degree $d=K_X^2 = 6-n$. These are the
  secondary and tertiary Burniat surfaces. There are four families of
  them, one for each of the four cases above.
\end{definition}

\section{Degree~$6$} 
\label{sec:deg6}

We refer the reader to \cite{kollar2023families-of-varieties} for the
definitions of log canonical and semi log canonical singularities and
of KSBA stable pairs. In brief, a KSBA pair consists of a projective
variety $X$ and a $\bQ$-divisor $B = \sum b_iB_i$ on it such that the
pair $(X,B)$ has semi log canonical singularities and the log
canonical divisor $K_X+B$ is ample. Once one fixes the basic invariants:
$\dim X$, the coefficient vector $(b_i)$ and the volume
$(K_X+B)^{\dim X}$, KSBA pairs admit a projective KSBA moduli space.

\subsection{The main theorem}

For primary Burniat surfaces, the main result of
\cite{alexeev2024explicit-compactifications} is the following theorem.
(Recall that the convention in the present paper is that $\oM_6\ubur$
is the \emph{normalization} of the closure of $M\ubur$ in $\oM\uslc$.)


\begin{theorem}\label{thm-intro:burniat}
  The KSBA compactification
  $\oM\ubur_6$ of the moduli space of primary Burniat surfaces is
  isomorphic to the quotient $\oM_6 / \Gamma_6$, where $\oM_6$ is the
  KSBA compactification of the moduli space of pairs
  $(Y_6, \sum_{i=0}^3 \frac12(R_i+G_i+B_i))$ shown in the middle
  panel of Figure~\ref{fig-burniat-config}, and the group
  $\Gamma_6=C_3\ltimes S_2^4$ acts by relabeling the curves.
  The space $\oM_6$ appears in the following diagram of intermediate spaces:
  \begin{displaymath}
    \oM_6\tor = \oM_6(\tfrac13) \xleftarrow{\rho_1} \oM_6(\tfrac25)
    \xleftarrow{\rho_2} \oM_6(\tfrac12) = \oM_6
  \end{displaymath}
  For $c=\frac13,\frac25,\frac12$ the variety $\oM_6(c)$ is the
  KSBA compactification of the moduli space of pairs $(Y_6,
  \sum_{i=0,3} \frac12(R_i+G_i+B_i) + \sum_{i=1,2} c(R_i+G_i+B_i))$.

  Further, $\oM_6\tor$ is a projective toric variety corresponding to
  an explicit fan $\fF_6$; $\rho_1$ is the blowup at one smooth point,
  the origin of the torus; and $\rho_2$ is the blowup at six disjoint
  smooth rational curves lying in the smooth locus.
  
  The map $\oM\ubur_6 \to \oM\uslc$ is a bijection to its image.  The
  surfaces $X_6$ appearing on the boundary of $\oM\ubur_6$ are
  $\bZ_2^2$-covers of the pairs
  $(Y_6, \sum_{i=0}^3 \frac12(R_i+G_i+B_i))$ parameterized by $\oM_6$.
  The boundary $\oM\ubur_6 \setminus~M\ubur_6$ is a union of eight
  divisors of types A, B, C, D, E, F, G, H corresponding to minimal
  degenerations, generically of these types, reviewed below.

  For $c=\frac13, \frac25, \frac12$ the moduli spaces $\oM_6(c)$ come with
  families of KSBA stable pairs
\begin{displaymath}
  f_6(c)\colon \big( \cY_6(c), \sum_{i=0,3} \tfrac12 (\cR_i+\cG_i+\cB_i) +
  \sum_{i=1,2} c(\cR_i+\cG_i+\cB_i) \big) \to \oM_6(c)
\end{displaymath}
\end{theorem}

Increasing the weight $c$ from
$\frac13$ to $\frac25$ to $\frac12$ changes the stability conditions
for the pairs, so it changes the moduli space and the family of pairs over
it.

    

\subsection{Degenerations}
\label{sec:degenerations-deg6}

We briefly describe the degenerations of types A, B, C, D, E, F, G, H
giving one example in a $\Gamma_6$-orbit. We refer
to \cite{alexeev2024explicit-compactifications} for complete details. 

\smallskip\noindent
(1) The degenerations A, B, C, D are toric, occurring when some of
  the interior curves degenerate to the boundary, see
  \cite[Figs.\ 3, 4, 6]{alexeev2024explicit-compactifications}:

\begin{enumerate}\label{page:ABCD}
\item[(A)] $R_1, R_2 \rightsquigarrow R_0 + G_3$, i.e. both $R_1$ and
  $R_2$ degenerate to $R_0+G_3$, while the ratio of the coordinates
  defining $R_1$ and $R_2$ stays invertible.
\item[(B)] $R_1 \rightsquigarrow R_0 + G_3$.
\item[(C)] $R_1 \rightsquigarrow G_0+R_3$ and $B_1\rightsquigarrow B_0+R_3$.
\item[(D)] $R_1 \rightsquigarrow R_0 + G_3$ and $B_1\rightsquigarrow R_0+B_3$.
\end{enumerate}

\smallskip\noindent (2) The degenerations E, F, G, H are nontoric.  
  
\smallskip\noindent (3) \label{page:H} In the degenerations B, G, H the limit
pairs $(Y_6,\frac12 D_6)$ are log canonical and do not change the
underlying surface $Y_6=\Sigma$.  One may say that these degenerations
are not ``essential''.  (Note, however, that the $\bZ_2^2$-covers
$X_6\to Y_6$ are no longer canonical.)
  
The case B was explained above.
The case G occurs when two curves in the same pencil coincide, e.g.
$R_1=R_2$; we denote this case $H_R$. The case H occurs when three
curves from the different pencils, e.g. $R_1,G_1,B_1$, pass through a
common point; we denote this case $H_{111}$.

\smallskip\noindent
The other degenerations do change the underlying surface $Y_6$.

\smallskip\noindent (4) \label{page:EF} The degenerations E and F occur when
$6$ (resp. $5$) of the interior curves pass through a common point. In
this case the surface $\Sigma$ degenerates to a union of two
irreducible components $(\bP^1\times\bP^1)\cup\bP^2$
(resp. $\Bl_1\Sigma\cup\bP^2$).  There is also a mixed degeneration EF
when the limit surface is $\Bl_1\Sigma\cup\bF_1\cup\bP^2$,
\cite[Figs. 10,11,12]{alexeev2024explicit-compactifications}.

\smallskip\noindent (5) All other degenerations, of types AA, AB, CDD etc., are
obtained as combinations of these minimal degenerations.

\subsection{The underlying surfaces}
\label{sec:surfaces-deg6}

In Figure~\ref{fig-deg6} we list all types of surfaces $Y_6$
underlying the stable pairs appearing in $\oM_6$. In these figures,
the bold curves indicate the double locus of the surfaces $Y_6$. The
irreducible components of $Y_6$ are glued together along the curves in
this double locus. Generically, the irreducible components of $X_6$
are in bijection with those of $Y_6$ and the double locus of $X_6$ is
the preimage of the double locus of $Y_6$.

The double curves are not part of the divisors $R$, $G$, $B$. The
colors indicate the ramification divisors of the $\bZ_2^2$-covers
$X_6\to Y_6$ when restricted to irreducible components of $Y_6$.  It
was difficult to draw the surfaces of types E and F in one picture, so
we show each as two parts; they are glued along the black (resp. red)
bold curve.  Similarly, a surface of type EF is the union of three
parts, shown separately.



  \medskip

In Figure~\ref{fig-irr6} we list the types of irreducible components
$V$ of $Y_6$ appearing in these pairs. The number in the parentheses
is the volume, defined as the degree $L^2$ of the restriction of the
following ample Cartier divisor to $V$:
  \begin{displaymath}
    L = 2\big(K_{Y_6} + \tfrac12 D_6 \big)|_{V},
    \quad\text{where} \  D_6 = \sum_{i=0}^3 (R_i+G_i+B_i).
  \end{displaymath}
  The volumes are positive integers and they add up to
  $4(K_{Y_6}+\frac12 D_6)^2 = K_{X_6}^2 = 6$.

Table~\ref{tab:deg6} lists the irreducible components for each surface
in Figure~\ref{fig-deg6}.



\subsection{The toric setup}
\label{sec:toric-setup-deg6}

Now fix the toric structure on the Cremona surface $\Sigma$, i.e. the
embedding into $\Sigma$ of the torus
\begin{displaymath}
  T_\Sigma= \Spec\bC[r,g,b] / (rgb-1)         \simeq (\bC^*)^2.
\end{displaymath}
Consider the configuration of the interior curves
$R_1,R_2,G_1,G_2,B_1,B_2$ as in the middle panel of
Figure~\ref{fig-burniat-config}. Let the interior curves of the same
color possibly coincide but assume that they do not degenerate to the
boundary. Then the locations of these curves are determined by a
$6$-tuple $(r_1,r_2,g_1,g_2,b_1,b_2)\in(\bC^*)^6$, where for $i=1,2$
the condition $r_i\to 0$ means $R_i\rightsquigarrow R_0 + G_3$ and
symmetrically for $G_i,B_i$.  Denote this $6$-dimensional torus by
$T_\cY$.

Acting by the torus $T_\Sigma$ on $T_\cY$, the six coordinates will change to
\begin{displaymath}
  (r_1,r_2,g_1,g_2,b_1,b_2) \to (rr_1,rr_2,gg_1,gg_2,bb_1,bb_2)
\end{displaymath}
So, up to the toric automorphism of $\Sigma$, the configuration
corresponds to a point in the four-dimensional quotient torus
\begin{displaymath}
  T_6 = T_\cY/T_\Sigma = \Spec \bC \left[ r_1g_1b_1 ,
  \frac{r_1}{r_2}, \frac{g_1}{g_2}, \frac{b_1}{b_2} \right].
\end{displaymath}

For the cocharacter lattices of these three tori, we get an exact sequence
\begin{displaymath}
  0 \to N_\Sigma \to N_\cY \to N_6 \to 0,
\end{displaymath}
where $N_\cY=\bZ^6$,
$N_\Sigma =
\{(\rho,\rho,\gamma,\gamma,\beta,\beta)\}$ with
$\rho+\gamma+\beta=0$, $N_6\subset\bZ^4$ is the set of
$4$-tuples $(\delta,\bar\rho,\bar\gamma,\bar\beta)$ such that
$\delta + \bar\rho + \bar\gamma + \bar\beta  = 0\pmod 2$, and the map
$N_\cY\to N_6$ is
\begin{displaymath}
  (\rho_1,\rho_2, \gamma_1,\gamma_2, \beta_1,\beta_2) \to
  (\rho_1+\rho_2+\gamma_1+\gamma_2+\beta_1+\beta_2,
  \rho_1-\rho_2, \gamma_1-\gamma_2, \beta_1-\beta_2).
\end{displaymath}

By \cite{alexeev2024explicit-compactifications}, the moduli space
$\oM_6\tor = \oM_6(\frac13)$ is a toric variety that is a
compactification of the torus $T_6$. Its fan $\fF_6$ has $42$ rays
with the integral generators that comprise the $\Gamma_6$-orbits of
the basic vectors
\begin{displaymath}
\text{A}\  (2,0,0,0),\qquad
\text{B}\  (1,1,0,0),\qquad
\text{C}\  (1,1,1,1),\qquad
\text{D}\  (0,1,0,1).
\end{displaymath}
under the relabeling group $\Gamma_6$ defined in the next section.
They correspond to the toric degenerations of types A, B, C, D listed
of Section~\ref{sec:degenerations-deg6}.

\subsection{$\bZ_2^2$-covers and the relabeling group}
\label{sec:covers-and-relabeling}

As we mentioned at the beginning of
Section~\ref{sec:burniat-surfaces}, the triple of divisors $R,G,B$ on
a a surface $Y$ defines a $\bZ_2^2$-Galois cover, the Burniat surface,
which is the main subject of this paper. The general theory in the
case of a smooth base $Y$ was given in
\cite{pardini1991abelian-covers}, and it was extended to the case of
semi log canonical covers $X\to Y$ in in
\cite{alexeev2012non-normal-abelian}. Thus, a stable pair
$(Y_6, \sum_{i=0}^3 \frac12(R_i+G_i+B_i))$ gives us a stable Burniat
surface $\pi\colon X_6\to Y_6$. A cyclic permutation of the RGB colors
corresponds to choosing a different basis in $\bZ_2^2$: R, G, B stand
for the three nonzero elements of $\bZ_2^2\setminus~0$.
(Only even permutations in $S_3$ are allowed by the symmetry of
Figure~\ref{fig-burniat-config}.)
This choice
does not change the isomorphism class of the cover $X_6$. Similarly,
switching $R_1\leftrightarrow R_2$ or $G_1\leftrightarrow G_2$ or
$B_1\leftrightarrow B_2$ would not change the sums $R,G,B$, so it
would not change $X_6$. Finally, the Cremona surface $\Sigma$ has the
Cremona involution which preserves the divisors $R,G,B$ while sending
$R_0\leftrightarrow R_3$, $G_0\leftrightarrow G_3$,
$B_0\leftrightarrow B_3$ simultaneously (we do not relabel the
interior curves). Thus, to obtain the moduli space of the surfaces
$X_6$, we have to divide the moduli space of the pairs
$(Y_6, \sum_{i=0}^3 \frac12(R_i+G_i+B_i)$ by a relabeling group
$\Gamma_6=(C_3\ltimes S_2^3)\times S_2$, of order $48$.

\begin{notations}\label{not:relabeling-group}
  We fix the notations for the action of the relabeling group
  $\Gamma_6=(C_3\ltimes S_2^3)\times S_2$ on the three tori:
  \begin{enumerate}
  \item The cyclic group $C_3$ cyclically permutes the RGB
    colors:
    \begin{displaymath}
      (r_1,r_2,g_1,g_2,b_1,b_2) \to
      (g_1,g_2,b_1,b_2,r_1,r_2) \to
      (b_1,b_2,r_1,r_2,g_1,g_2) \to\text{back}.
    \end{displaymath}
  \item $S_2\times S_2\times S_2$ permutes the interior curves of the
    same color. The involution $s_r$ switches 
    \begin{math}
      s_r\colon r_1\leftrightarrow r_2
    \end{math}
    while keeping the other four coordinates intact, and similarly for
    $s_g\colon g_1\leftrightarrow g_2$ and $s_b\colon
    b_1\leftrightarrow b_2$. 
    \item The final $S_2$ is the Cremona involution on $\Sigma$. It sends
    \begin{displaymath}
      \Cr\colon (r_1,r_2,g_1,g_2,b_1,b_2) \to
      (r_1\inv,r_2\inv,g_1\inv,g_2\inv,b_1\inv,b_2\inv).
    \end{displaymath}
  \end{enumerate}
\end{notations}

It is then easy to see how $\Gamma_6$ acts on the cocharacter
lattices. On $N_\cY$, $C_3$ cyclically permutes $\rho_i,\gamma_i,\beta_i$; $s_r$
permutes $\rho_1\leftrightarrow\rho_2$ and similarly for $s_g,s_b$;
and the Cremona involution sends
$(\rho_1,\rho_2,\gamma_1,\gamma_2,\beta_1,\beta_2)$ to
$(-\rho_1,-\rho_2,-\gamma_1,-\gamma_2,-\beta_1,-\beta_2)$.  The
induced action of $\Gamma_6$ on $N_6$ is as follows:
\begin{enumerate}
\item $C_3$ cyclically permutes $\bar\rho$, $\bar\gamma$, $\bar\beta$;
\item The involution $s_r$ sends
  $(\delta, \bar\rho, \bar\gamma, \bar\beta)$ to
  $(\delta, -\bar\rho, \bar\gamma, \bar\beta)$, and similarly for
  $s_g$, $s_b$.
\item The Cremona involution sends
  $(\delta, \bar\rho, \bar\gamma, \bar\beta)$ to
  $(-\delta, -\bar\rho, -\bar\gamma, -\bar\beta)$.
\end{enumerate}

\smallskip

Reflecting the degenerations of types E and F of
Section~\ref{sec:degenerations-deg6}(4), the morphism
$\rho_1\colon \oM_6(\frac25)\to\oM_6\tor$ is the blowup at the origin
$(1,1,1,1) \in T_6$, and the morphism
$\rho_2\colon \oM_6 = \oM_6(\frac12)\to \oM_6(\frac25)$ is the further
blowup at the disjoint union of six smooth rational curves, the strict
preimages of the closures of the six $\bC^*\subset T_6$ of the form
$(x,1,1,1,1,1) \mod T_\Sigma$, $(1,x,1,1,1,1) \mod T_\Sigma$, etc.,
with $x\in\bC^*$. Using the nonsymmetric coordinates $r_1g_1b_1$,
$r_1/r_2$, $g_1/g_2$, $b_1/b_2$ on $T_6$, they are
\begin{equation}\label{eq:six-curves}
  (x,x,1,1),\ (1,x\inv,1,1),\ (x,1,x,1),\ (1,1,x\inv,1),\
  (x,1,1,x),\ (1,1,1,x\inv).
\end{equation}

\section{Evolution of  irreducible components}
\label{sec:evolution}

Let $V$ be one of the irreducible components of a del Pezzo surface
$Y_6$ that appears in Figure~\ref{fig-irr6}, and denote its boundary
(pictured by bold lines) by $\Delta$.
Suppose that some of the
triples of the RGB divisors are incident at common points
$P_{i_sj_sk_s} = R_{i_s}\cap G_{j_s}\cap B_{k_s}$, $s=1,\dotsc,
n$. Let $V'\to V$ be a blowup at this point and $R'_i$, $G'_i$, $B'_i$
be the strict preimages of $R_i$, $G_i$, $B_i$. Further assume that
the doubled log canonical divisor
\begin{displaymath}
L' = 2 \big(K_{V'} + \Delta + \sum_{i=0}^3 \tfrac12 (R'_i+G'_i+B'_i)\big)
\end{displaymath}
is nef. It follows from the standard results of the Minimal Model
Program (and is very easy to check anyway for these surfaces) that
some multiple $|mL'|$ defines a contraction $V'\to V'\can$ to the
canonical model. If $L'$ is big then the pair
$(V'\can, \Delta\can + \frac12(R'\can + G'\can + B'\can))$ has log
canonical singularities and ample log canonical class; in other words
it is an irreducible KSBA stable pair.

\begin{notation}\label{not:new-components}
  If a surface is obtained from one of the surfaces $\hash m(v)$ of
  degree~6 of Figure~\ref{fig-irr6} with volume $L^2=v$ by making $n$
  blowups and then contracting by the nef line bundle $L'$, we denote
  it by $\hash m_n(v-n)$.
  Here, $v-n = L'{}^2$ is the volume of the new surface.

\end{notation}

The above procedure is exactly how irreducible components of primary
stable Burniat surfaces give birth to those of secondary and and
tertiary Burniat surfaces. A surface of type $\hash m_n(v-n)$ is a
child of a surface of type $\hash m(v)$.  In
Figure~\ref{fig-new-irr} we list and name them. Note that
$\hash2_1(1) = \hash4(1)$ and the first kind of
$\hash 6_2(2) = \hash5(2)$ are not new, they have already appeared in
Figure~\ref{fig-irr6}.

Consider now the case when the surface $\hash8(1)$, isomorphic to
$\bP^2$, is blown up once. The volume goes down by one, so using the
above notation, we would call this case $\hash8_1(0)$. One has
$V'=\Bl_1\bP^2\simeq\bF_1$. The divisor $L'$ in this case is not
big. But is nef and the linear system $|L'|$ defines a contraction
$\bF_1\to\bP^1$ along the ruling. When this situation occurs in the
following sections, the entire component is contracted to a $\bP^1$.

A variation of this is the component $\hash8(1)$ blown up twice, a
case that we could denote $\hash8_2(-1)$. In this case $L'$ is no longer
nef and there is a $(-1)$-curve $C$, the strict preimage of a line
through the two blown up points such that $L'\cdot C=-1$. In fact,
$L'=C$. Let $f\colon V'\to V''$ be the contraction of $C$ to a
point. We have $L''= f_*L' = 0$, so the linear system $|L''|$ defines
the contraction to a point. When this situation occurs in
Section~\ref{sec:deg4b}, it leads to a flip. The entire component is
contracted, and a neighboring irreducible component is modified: a
point on it is blowup up.

\section{Degree $3$}
\label{sec:deg3}

The method of the proof of
Theorems~\ref{thm-intro:5}--\ref{thm-intro:3} is essentially the same,
with some variations. We start with the easiest case of degree~3 where
there are few degenerations to obscure the general picture. Consider
the closed subset
\begin{displaymath}
  Z_3 := H_{211}\cap H_{121} \cap H_{112} \subset \oM_6,
\end{displaymath}
a projective variety that is the intersection of three type H divisors
corresponding to the type H degenerations defined in
Section~\ref{sec:degenerations-deg6}(3). For example, the divisor
$H_{211}$ corresponds to the closed subset of $\oM_6$ where on every
fiber the intersection $R_2\cap G_1\cap B_1$ is nonempty, necessarily
is a point, which we denote $P_{211}$.

\smallskip\emph{Step 1.} Let us describe $Z_3$ explicitly. Its image
under the morphism $\oM_6\to\oM_6\tor$ is the toric variety $Z_3\tor$
which is the closure in $\oM_6\tor$ of the subtorus $T_3\subset T_6$
defined by the equations
\begin{displaymath}
  T_3 = \{r_2g_1b_1 = r_1g_2b_1 = r_1g_1b_2 = 1\} .
\end{displaymath}
It easily follows that $T_3 = \Spec\bC[r_1/r_2]$ and that
$r_1/r_2 = g_1/g_2 = b_1/b_2$. The cocharacter lattice of $T_3$ is
\begin{math}
  N_3 = \bZ(-1,1,1,1) \subset N_6.  
\end{math}

The fan $\fF_3$ of $Z_3\tor$ is the intersection $\fF_6\cap N_3$. Both
$\bR_{\ge0}(-1,1,1,1)$ and $\bR_{\ge0}(1,-1,-1,-1)$ are rays of type C
in $\fF_6$, so they give two rays of type C in $\fF_3$. We see that
$Z_3\tor = \bP^1$, with the points $0$ and $\infty$ corresponding to
degenerations of type C of Section~\ref{sec:degenerations-deg6}. The
variety $Z_3$ is the strict preimage of $Z_3\tor$ under the blowups
$\rho_1, \rho_2$ of Theorem~\ref{thm-intro:burniat}, so
$Z_3 = Z_3\tor = \bP^1$ and the point $1\in\bC^* \subset \bP^1$
corresponds to a divisor of type E.

\smallskip\emph{Step 2.} Now restrict the universal family
\begin{displaymath}
  f_6(\tfrac12)\colon\cY_6(\tfrac12) \to \oM_6(\tfrac12)=\oM_6  
\end{displaymath}
of Theorem~\ref{thm-intro:burniat} to $Z_3$. We get a family
$\cY_6 |_{Z_3}\to Z_3$ in which every fiber is a pair
$(Y_6,\sum_{i=0}^3\frac12(R_i+G_i+B_i))$ with three distinguished points
\begin{displaymath}
  P_{211} = R_2\cap G_1\cap B_1, \ 
  P_{121} = R_1\cap G_2\cap B_1, \
  P_{112} = R_1\cap G_1\cap B_2.
\end{displaymath}
Generically, $Y=\Sigma$ is the Cremona surface and the pairs are those
appearing in the last panel of Figure~\ref{fig-secondary-burniat} in
the definition of tertiary Burniat surfaces.

\begin{lemma}\label{lem:n-pts}
  Let $Z = \cap_{s=1}^n H_{i_sj_sk_s} \subset\oM_6$ be an intersection
  of $n$ type H divisors. Then in every fiber $Y_6$ of the restricted
  family $\cY_6|_Z\to Z$ the $n$ points $P_{i_sj_sk_s}$ are distinct
  and lie in the smooth part of $Y_6$.
\end{lemma}
\begin{proof}
  This follows from the condition that each fiber
  $(Y, \sum_{i=0}^3 \frac12 (R_i+G_i+B_i))$ is semi log canonical.
  Indeed, if two of the points coincide then at least five of the
  curves $R_i,G_i,B_i$ pass through the same point; then the
  discrepancy of the log canonical divisor at this point is
  $\le 1-\frac52 < -1$, which contradicts semi log
  canonicity. Similarly, a point $P_{i_sj_sk_s}$ can not be contained
  in the double locus, since then the discrepancy at this point would
  be $\le 1-1-\frac32 < -1$. Finally, we observe that each surface in
  Figures~\ref{fig-deg6} is nonsingular outside of the double
  locus. One can see this from Figure~\ref{fig-irr6} which lists the
  irreducible components.
\end{proof}

So, in our case the points $P_{211}, P_{121}, P_{112}$ are smooth and
distinct. By examining the list of surfaces in Figure~\ref{fig-deg6},
we see that the only possibilities with $Y_6\ne\Sigma$ are the three
surfaces depicted in Figure~\ref{fig-deg3}.

The first and the third surfaces $Y_6$ have three irreducible
components of type $\hash 2(2)\,[\bP^1\times\bP^1]$ of
Figure~\ref{fig-irr6}, and the second one has two irreducible
components, of types $\hash5(2)\,[\bP^1\times\bP^1]$ and
$\hash6(4)\,[\bP^2]$.

Now let $\cY'_3$ be the blowup of $\cY_6|_{Z_3}\to Z_3$ along the
three sections corresponding to the points $P_{211},P_{121},P_{112}$,
and let $\cR'_i, \cG_i', \cB_i'$ be the strict preimages of the
divisors $\cR_i, \cG_i, \cB_i$ on $\cY_3$.  The fibers of the family
$\cY'_3\to Z_3$ are the blowups of the corresponding surfaces $Y_6$ at
three points. By Section~\ref{sec:evolution}, 
the doubled log canonical divisor
\begin{displaymath}
  \cL' = \cO_{\cY'_3} \big(2 (K_{\cY'_3} + \sum_{i=0}^3
  (\cR_i+\cG_i+\cB_i) ) \big)
\end{displaymath}
is relatively nef and big over $Z_3$. By the Base Point Free theorem,
a standard result in the Minimal Model Program, some big multiple
$|m\cL'|$ defines a contraction $\cY_3'\to \cY_3'\can$ to the family
of canonical models of the fibers
$(Y', \sum_{i=0}^3\frac12(R'_i+G'_i+B'_i))$, relative over $Z_3$. The
fibers have semi log canonical singularities and ample log canonical
class $K + \sum_{i=0}^3\frac12(R_i+G_i+B_i)$, so they are KSBA stable
pairs.

Semi log canonical varieties are seminormal, so to describe the fibers
of $\cY_3'\can\to Z_3$, it is enough to find their irreducible
components. We have described those in
Section~\ref{sec:evolution}. They are $\hash2_1(1) = \hash4(1)$,
$\hash5(2)$ and $\hash6_3(1)$.  Over $Z_3\setminus \{0,1,\infty\}$ the
fiber is $\hash0_3(3)$, a del Pezzo surface of degree $3$ with three
$A_1$-singularities.  We picture the stable pairs in
Figure~\ref{fig-final3} and list their irreducible components in
Table~\ref{tab:deg3}.  We set $\oM_3 := Z_3$, this is our moduli space
of stable labeled pairs $(Y_3, \sum_{i=0}\frac12(R_i+G_i+B_i))$.

\smallskip\emph{Step 3.}
At this point, we have constructed a fine moduli space $\oM_3$ of labeled pairs
which comes with a universal family
\begin{displaymath}
  f\colon (\cY_3, \sum_{i=0}^3 \tfrac12(\cR_i+\cG_i+\cB_i))\to \oM_3
\end{displaymath}
As we reviewed in Section~\ref{sec:covers-and-relabeling}, for each
pair in this family the triple of divisors $R, G, B$ defines a
$\bZ_2^2$-cover $\pi\colon X_3\to Y_3$. The map $\pi$ is equivalent to
the surface $Y_3$ together with a triple of divisors, and the color of
each part does not matter. Thus, in order to get the moduli stack of
covers we must divide $\oM_3$ by the action of a relabeling group
$\Gamma_3$ permuting the curves $R_i,G_i,B_i$ is such a way that the
partition into three groups (arbitrarily colored) is preserved. This
leads to the Deligne-Mumford quotient stack $[\oM_3 : \Gamma_3$.

The moduli stack of the surfaces $X_3$ themselves is a $\bZ_2^2$-gerbe
over $[\oM_3 : \Gamma_3]$, since each of them has an additional
automorphism group $\bZ_2^2$ which the pair $(Y_3, \frac12(R+G+B)$
does not have. This is just a general fact about moduli stacks of
covers $X$ vs the moduli stack of the base varieties $Y$ together with
the covering data.  It is also important to note that for a smooth
Burniat surface $X$, the cover $\pi\colon X\to Y$ is intrinsic, as it
is given by the bicanonical linear system, cf. \cite[Cor.\
3.4]{bauer2010burniat-II}.

We now find the relabeling group $\Gamma_3$.  Generically, the surface
$Y'_3$ is a weak del Pezzo which has three $(-2)$-curves
$\{R_1,G_1,B_1\}$ and twelve $(-1)$-curves, nine of which are the
remaining branch curves.  $R_i,G_i,B_i$.
The sets $\{R_2,G_2,B_2\}$ and
$\{R_0,R_3,G_0,G_3,B_0,B_3\}$ do not mix: they are distinguished by
how many other $(-1)$-curves each of these curves intersects.  It
follows that the group of possible relabelings is
\begin{displaymath}
  \Gamma_3 = \Stab_{\Gamma_6} \{ (R_2,G_1,B_1), \ (R_1,G_2,B_1),\
  (R_1,G_1,B_2) \} \subset \Gamma_6.
\end{displaymath}
(The triples $(R_i,G_j,B_k)$ are unordered.)  It is easy to see that
$\Gamma_3 = C_3 \times S_2$ with $C_3$ cyclically permuting the colors
RGB and with $S_2 = \la \Cr\ra$ generated by the Cremona
transformation (cf. Notations~\ref{not:relabeling-group}). This gives
a family of surfaces $X$ over the $\bZ_2^2$-gerbe over the quotient
stack
\begin{displaymath}
  [\oM_3 : \Gamma_3] = [\bP^1 : (C_3\times S_2)].
\end{displaymath}

Moreover, $C_3$ acts trivially on the torus $T_3$ and thus also on the
variety $\oM_3$: it permutes $r,g,b$, but the coordinate on $T_3$ is
$r_1/r_2 = g_1/g_2 = b_1/b_2$.  This $C_3$ only contributes to the
automorphism groups of surfaces $Y_3$ together with three divisors and
of the covers $X_3$ but not to the isomorphism classes.  The Cremona
involution acts nontrivially: it sends $r_1/r_2$ to
$(r_1/r_2)\inv$. Thus, the coarse moduli space of the above stack is
$\bP^1/S_2\simeq \bP^1$, identifying the two divisors of type C.

For the coarse moduli spaces, we have the classifying map
$\oM_3/S_2 \to \oM\slc$ to the moduli space of stable KSBA stable
surfaces. On the open subset $\oM_3\setminus~\{0,\infty,1\} / S_2$ it
is an isomorphism to the image, the moduli space $M_3\ubur$ of smooth
tertiary Burniat surfaces.  It follows that $\oM_3/S_2= \oM\ubur_3$,
the KSBA compactification of $M\ubur_3$, since we defined the latter
to be the normalization of the closure of $M\ubur_3$ in
$\oM\slc$. Finally, we enumerate the boundary divisors in
$\oM_3\ubur$: there are one of type C and one of type E. This
completes the proof of Theorem~\ref{thm-intro:3}.

\begin{remark}
  By analyzing deformations of surfaces $X_3$, one could prove that in
  fact the closure of $M\ubur_3$ in $\oM\slc$ is $\oM_3/S_2$, even
  without normalizing.  A similar analysis for the degrees $4a,4b,5$
  becomes quite difficult.
\end{remark}

\section{Degree $4$ non-nodal}
\label{sec:deg4a} 

To prove Theorem~\ref{thm-intro:4a}, we repeat the same three steps of
the proof of Theorem~\ref{thm-intro:3}, modifying the details appropriately.
We fix the closed subset
\begin{displaymath}
  Z\foura := H_{111} \cap H_{222} \subset \oM_6,
\end{displaymath}
a projective variety that is an intersection of two type H divisors. 

\smallskip\emph{Step 1.} The image of $Z\foura$ under the morphism
$\oM_6\to \oM_6\tor$ is the toric variety $Z\foura\tor$ which is the
closure in $\oM_6\tor$ of the subtorus $T\foura \subset T_6$,
\begin{displaymath}
  T\foura = \{r_1g_1b_1 = r_2g_2b_2 = 1\}
  = \Spec \bC\left[ \frac{r_1}{r_2}, \frac{g_1}{g_2}, \frac{b_1}{b_2}\right] /
  \left( \frac{r_1}{r_2}\cdot \frac{g_1}{g_2}\cdot \frac{b_1}{b_2} - 1\right).
\end{displaymath}
Its cocharacter lattice is
\begin{displaymath}
  N\foura = \{\delta+\bar r+\bar g+\bar b = \delta-\bar r-\bar g-\bar b=0\}
  = \{ \delta = \bar r + \bar g + \bar b = 0\} \subset N_6.
\end{displaymath}
We computed the fan $\fF\foura=\fF_6\cap N\foura$ of $Z\foura\tor$
using the sage script attached to the version of this paper on
arXiv. (The lattice $N\foura$ is small enough to do it by hands as
well.) The result is:

\begin{lemma}
  One has the following:
  \begin{enumerate}
  \item The fan $\fF\foura$ is a nonsingular fan with $6$ rays of type D
    with integral generators
    \begin{displaymath}
      \pm(0,1,-1,0),\ \pm(0,0,1,-1),\ \pm(0,-1,0,1)  
    \end{displaymath}
    and $6$ cones of type DD.
  \item $Z\foura\tor$ is isomorphic to the Cremona surface
    $\Sigma=\Bl_3\bP^2$. 
  \end{enumerate}
\end{lemma}

The six curves of Equation~\eqref{eq:six-curves} intersect
$Z\foura\tor$ transversally in $T_6$, meeting at $(1,1,1,1)$ -- the
center of the blowup $\rho_1$. Thus, their strict preimages -- the
centers of the blowup $\rho_2$ -- are disjoint from
$\rho_*\inv(Z\foura\tor)$ and one has $Z\foura = \Bl_1Z\foura\tor$. The
exceptional divisor corresponds to the divisor of type E. The sage
script also computes the full stratification of $Z\foura$, which in
addition to the D, DD and E strata has strata of types G, DG and
EG. (Again, this is easy to do by hands.)

\smallskip\emph{Step 2.} We find all the possible fibers of the
restricted family $\cY_6|_{Z\foura} \to Z\foura$ by examining the
surfaces of Figure~\ref{fig-deg6} and picking those where it is
possible to have the points $P_{111} = R_1\cap G_1\cap B_1$ and
$P_{222} = R_2\cap G_2\cap B_2$ which must be smooth and distinct by
Lemma~\ref{lem:n-pts}. These possibilities are listed in
Figure~\ref{fig-deg4a}. We omit the generic case of $Y_6=\Sigma$.  

Let $\cY'\foura$ be the blowup of $\cY_6|_{Z\foura}$ along the two
disjoint sections corresponding to the points $P_{111}, P_{222}$.
It follows from the computations of Section~\ref{sec:evolution} 
that $\cL'$, the doubled log canonical divisor,
is relatively big and nef over $Z\foura$. Let $\cY'\to\cY'\can$ be
the contraction to the relative canonical model over $Z\foura$. It is
a family of KSBA stable pairs $(Y\foura,\sum_{i=0}^3\frac12(R_i+G_i+B_i))$.

Figure~\ref{fig-new-irr} shows the new components $\hash3_1(2)$ and
$\hash6_2(2)$; the latter turns out to be isomorphic to
$\hash5(2)$.
(Note: the DG case gives
the same surface as in the D case, and the EG case the same surface as
in the E case.) Over an open subset of $Z\foura$, the fiber is a del
Pezzo surface of degree $4$ which we denote $\hash0_2(4)$. On the G divisor, it
acquires an $A_1$ singularity; in this case it is a nodal del Pezzo.

Figure~\ref{fig-final4a} shows the fibers in the KSBA family
$\cY\foura \to Z\foura$, and Table~\ref{tab:deg4a} lists their types.

We observe that the fibers in the family $\cY\foura^{\rm can}\to
Z\foura$ of labeled pairs are all distinct with one exception: over
the E divisor, the exceptional divisor of the blowup $\rho_1\colon
\Bl_1\Sigma \to \Sigma$,  all fibers -- of type E -- are isomorphic. Let
$\oZ\foura = \Sigma$ and let $Z\foura \to \oZ\foura$ be the contraction
back to $\Sigma$. The family
$\cY\foura^{\rm can}\to Z\foura$
descends to a family 
$\overline{\cY}\foura^{\rm can}\to \oZ\foura$ in which every labeled
pair appears only once. We set $\oM\foura := \oZ\foura$.

\smallskip\emph{Step 3.} As in the previous section, we have to
divide $\oM\foura$ by the relabeling group $\Gamma\foura$. One part of
this group is obvious, it is 
\begin{displaymath}
  \Gamma\foura := \Stab_{\Gamma_6} \{ (R_1,G_1,B_1), \ (R_2,G_2,B_2) \}
  \simeq C_3\times S_2^2 = C_3 \times \la s_r\circ
s_g\circ s_b\ra \times \la\Cr\ra,
\end{displaymath}
using Notations~\ref{not:relabeling-group}, with $C_3$ cyclically
permuting the colors. But the full relabeling group is twice as big:
$\wt\Gamma\foura = \Gamma\foura\rtimes S_2$. One choice for $S_2$ is
$\la \sigma\ra$ with
\begin{displaymath}
  \sigma = (R_0R_1)(R_3R_2)(G_0B_1)(G_3B_2)(B_0G_1)(B_3G_2),
\end{displaymath}
which swaps the two groups $\{R_0,R_3,G_0,G_3,B_0,B_3\}$
$\{R_1,R_2,G_1,G_2,B_1,B_2\}$.

The automorphism group of a generic degree~$4$ del Pezzo surface is
$A \simeq S_2^4$. A direct computation shows that
\begin{math}
  A\foura := A\cap \wt\Gamma\foura = \operatorname{Center}(\wt\Gamma\foura) = S_2,
\end{math}
generated by the involution
\begin{displaymath}
  s_r \circ s_g\circ s_b\circ \Cr  = 
   (R_0R_3)(G_0G_3)(B_0B_3)(R_1R_2)(G_1G_2)(B_1B_2).
\end{displaymath}
We note that $A\foura\subset\Gamma\foura$, so $\sigma$ does not define
an automorphism of a generic surface $Y\foura$ in our family. Thus,
$\sigma$ defines an alternative choice for a coordinate system in
$M\foura$.  There are exactly four $(-1)$ curves that are not branch
divisors, which we can denote $E_{000}, E_{111}, E_{222},
E_{333}$. Contracting $E_{111}, E_{222}$ gives a morphism
$Y_3\to \Sigma$ leading to $\{R_0,R_3,G_0,G_3,B_0,B_3\}$ being the set
of external curves, and contracting $E_{000}, E_{333}$ gives a
morphism $Y_3\to \Sigma$ to a different Cremona surface, with the
external curves $\{R_1,R_2,G_1,G_2,B_1,B_2\}$.

These two presentations of $Y_4$ as $\Bl_2\Sigma$ are equivalent. A
direct computation shows that the coordinate systems in the two
presentations are related by
\begin{displaymath}
  \frac{r_0}{r_3} = \frac{g_2}{g_1}, \quad
  \frac{b_0}{b_3} = \frac{r_2}{r_1}, \quad
  \frac{g_0}{g_3} = \frac{b_2}{b_1}.
\end{displaymath}
One can check that these coordinate systems extend to the
identification of the compactifications $\oM\foura$ as well.

Thus, the moduli stack of labeled pairs modulo relabeling is
$[\oM\foura : \Gamma\foura]$. The subgroup $A\foura\simeq S_2$ acts
trivially on $T_4$ and thus also on $\oM\foura$, and the quotient
group $\Gamma\foura / A\foura = C_3\times S_2$ acts faithfully.
As in the previous section, we obtain a $\bZ_2^2$-gerbe over
$[\oM\foura : \Gamma\foura]$
whose coarse moduli space $\oM\foura / 
(C_3\times S_2)$ is the normalization of the closure of
$\oM\ubur\foura$ in $\oM\slc$. 

It is interesting to note that for the type E surface of
Figure~\ref{fig-final4a} the group of automorphisms preserving a
partition into three groups is $\Gamma\foura$. This is 
consistent with the above description of a moduli stack, since
$1 \in \bC^* \subset\oM\foura$ is fixed by $\Gamma\foura$.

Finally, we enumerate the boundary strata. There is one type D
divisor. The divisor of type E was contracted to a point, and the
surfaces over the G divisor are canonical, so are not counted as part
of the boundary.  This concludes the proof of
Theorem~\ref{thm-intro:4a}.

\section{Degree $4$ nodal}
\label{sec:deg4b} 

Here, we fix the closed subset
\begin{displaymath}
  Z\fourb := H_{111} \cap H_{122} \subset \oM_6,
\end{displaymath}
a projective variety that is an intersection of two type H divisors. 

\smallskip\emph{Step 1.} The image of $Z\fourb$ under the morphism
$\oM_6\to \oM_6\tor$ is the toric variety $Z\fourb\tor$ which is the
closure in $\oM_6\tor$ of the subtorus $T\fourb\subset T_6$,
\begin{displaymath}
  T\fourb = \{r_1g_1b_1 = r_1g_2b_2 = 1\}
  = \Spec \bC\left[ \frac{r_1}{r_2}, \frac{g_1}{g_2}, \frac{b_1}{b_2}\right] /
  \left( \frac{g_1}{g_2}\cdot \frac{b_1}{b_2} - 1\right).
\end{displaymath}
Its cocharacter lattice is
\begin{displaymath}
  N\fourb = \{\delta+\bar r+\bar g+\bar b = \delta+\bar r-\bar g-\bar b=0\}
  = \{ \delta + \bar r = \bar g + \bar b = 0\} \subset N_6.
\end{displaymath}
We computed the fan $\fF\fourb=\fF_6\cap N\fourb$ of $Z\fourb\tor$
using the sage script attached to the version of this paper on
arXiv. (Again, the lattice $N\fourb$ is small enough to do it by
hands.) The result is given in Lemma~\ref{lem:F4b}.

\begin{lemma}\label{lem:F4b}
  One has the following:
  \begin{enumerate}
  \item The fan $\fF\fourb$ is a nonsingular fan with $8$ rays, two of
    type B, four of type C and two of type D, with the integral generators
    \begin{displaymath}
      B:\ \pm(1,-1,0,0), \quad
      C:\ \pm(1,-1,1,-1),\ \pm(1,-1,-1,1), \quad
      D:\ \pm(0,0,1,-1).
    \end{displaymath}
    It also has $8$ two-dimensional cones, four each of types BC and CD.
  \item $Z\fourb\tor$ is a nonsingular toric variety isomorphic to the
    blowup of $\bP^1\times\bP^1$ at the four torus-fixed points.
  \end{enumerate}
\end{lemma}

The first blowup $\rho1{}_*\inv Z\fourb\tor$ is the blowup of this toric
variety $Z\fourb\tor$ at the origin. 
Of the six curves of Equation~\eqref{eq:six-curves}, there is one
curve that lies in $T\fourb$: 
\begin{equation}\label{eq:F-curve}
  \{ (1,x,1,1,1,1)\in T_\cY/ T_\Sigma,\ x\in\bC^*\} =
  \{(1,x\inv,1,1) \in T_6, x\in\bC^* \}
\end{equation}
corresponding to the degeneration of type F where in the non-stable
limit the curves $R_1,G_1,G_2,B_1,B_2$ pass through a single point
(Section~\ref{sec:degenerations-deg6}(4)). In $\oM_6$ it is replaced
by a KSBA stable pair. The other five curves intersect $T\fourb$
transversally at the origin. It follows that
$Z\fourb = \rho_*\inv Z\fourb\tor$, the blowup of the toric variety
$Z\fourb\tor$ at the origin, and that it contains a single divisor
$D_F\simeq\bP^1$, the closure of the curve~\eqref{eq:F-curve}.

\smallskip\emph{Step 2.} We find all the possible fibers of the
restricted family $\cY_6|_{Z\fourb} \to Z\fourb$ by examining the
surfaces of Figure~\ref{fig-deg6} and picking those where it is
possible to have the points $P_{111} = R_1\cap G_1\cap B_1$ and
$P_{122} = R_1\cap G_2\cap B_2$ which must be smooth and distinct by
Lemma~\ref{lem:n-pts}. These possibilities are listed in
Figure~\ref{fig-deg4a}. We omit the generic case of $Y_6=\Sigma$. We
do not list the B, G, H divisors which do not change the underlying
surface.

Let $\cY'\fourb$ be the blowup of $\cY_6|_{Z\fourb}$ along the two
disjoint sections corresponding to the points $P_{111}, P_{122}$. One
easily checks that $\cL'$, the doubled relative log canonical divisor,
is big and nef on all fibers over $Z\fourb$ \emph{except} for those in
the divisor $D_F$ of type F defined above. In this latter case the
surface is of the blowup of the surface of type $\hash8(1)\, [\bP^2]$
and has the volume $L'{}^2 = 1-2=-1$, so $L'$ is not big.
We noted this special case at the end of Section~\ref{sec:evolution}.

We correct the family $\cY'\fourb\to Z\fourb$ by making a flip
$\cY\fourb'\ratmap\cY\fourb''$. Luckily, it is of the easiest possible
type: it is the usual Atiyah flop (for $K_{\cY'\fourb}$) relative over
$Z\fourb$. The curves $C$ in the fibers over $D_F$ with
$K_{\cY'\fourb} \cdot C =-1$ sweep a codimension-$2$ subset $\cC$
lying in the smooth locus of $\cY'\fourb$. The map $\cC\to D_F$ is a
$\bP^1$-fibration. We blow it up to produce a
$\bP^1\times\bP^1$-fibration and contract it in the other direction to
produce the variety $\cY''\fourb$ containing $\cC^+$, which is another
$\bP^1$-fibration over $D_F$. The effect on the fibers is the
contraction of the curve $C$ followed by blowing up a point on the
neighboring surface producing a curve $C^+$. To a reader familiar with
degenerations of K3 surfaces, this will look exactly like 
a type II modification of Kulikov models.

On the new family $\cY''\fourb\to Z\fourb$ the relative canonical
class is nef, following the computation in
Section~\ref{sec:evolution}. Its multiple gives the contraction
$\cY''\fourb \to \cY_{4b}''{}\can$ to the relative canonical model, a
family of KSBA-stable pairs. The resulting KSBA stable pairs are
pictured in Figure~\ref{fig-final4b} and their irreducible components
are listed in Table~\ref{tab:deg4b}. Over an open subset of $Z\fourb$
the fibers are degree~$4$ del Pezzo surfaces with two
$A_1$-singularities. Over the F divisor they acquire an $A_2$
singularity.

In the family $\cY_{4b}''{}\can \to Z\fourb$ 
the fibers, labeled stable pairs $(Y\fourb, \sum_{i=0}^3\frac12
(R_i+G_i+B_i))$ are all pairwise non-isomorphic. We set $\oM\fourb :=
Z\fourb$. 

\smallskip\emph{Step 3.} In this case, the relabeling group is
\begin{displaymath}
  \Gamma\fourb := \Stab_{\Gamma_6} \{ (R_1,G_1,B_1), \ (R_1,G_2,B_2) \}
  \simeq S_2\times S_2 =
  \la s_g\circ s_b\ra \times \la\Cr\ra,
\end{displaymath}
It acts faithfully on $T\fourb$ and thus also on $\oM\fourb$. 

Arguing as in the previous sections, we obtain a $\bZ_2^2$-gerbe over
the quotient stack $[Z\fourb : \Gamma\fourb]$ whose coarse moduli
space $Z\fourb / S_2^2$ has a finite map to the closure of
$M\ubur\fourb$ in $\oM\slc$. Thus, $\oM\ubur\fourb = Z\fourb/S_2^2$.

Finally, we enumerate the boundary divisors. There are one of each
types B, C, D, E, G. The surfaces over the F divisor are canonical, so
they are not counted as part of the boundary.  This completes the
proof of Theorem~\ref{thm-intro:4b}.

\section{Degree $5$}
\label{sec:deg5} 

Here, we fix the closed subset
\begin{math}
  Z_5 := H_{111} \subset \oM_6,
\end{math}
a type H divisor.

\smallskip\emph{Step 1.} The image of $Z_5$ under the morphism
$\oM_6\to \oM_6\tor$ is the toric variety $Z_5\tor$ which is the
closure in $\oM_6\tor$ of the subtorus $T_5\subset T_6$,
\begin{displaymath}
  T_5 = \{r_1g_1b_1 \}
  = \Spec \bC\left[ \frac{r_1}{r_2}, \frac{g_1}{g_2}, \frac{b_1}{b_2}\right].
\end{displaymath}
Its cocharacter lattice is
\begin{displaymath}
  N_5 = \{\delta+\bar r+\bar g+\bar b \} \subset N_6.
\end{displaymath}
We computed the fan $\fF_5=\fF_6\cap N_5$ of $Z_5\tor$
using the sage script attached to the version of this paper on
arXiv. The result is given in Lemma~\ref{lem:F5} below. But first, we
introduce some notations.

\begin{notations}\label{not:F5}
  Let us introduce several rays of types B, C, D in $N_5$ with the
  following integral generators
  \begin{eqnarray*}
    &&\text{C:}\quad e_1=(1,1,-1,-1)\quad e_2=(1,-1,1,-1)\quad e_3=(1,-1,-1,1)\\
    &&\text{B:}\quad \frac{e_1+e_2}{2} = (1,0,0,-1) \qquad
    \text{D:}\quad \frac{e_1-e_2}{2} = (0,1,-1,0)
  \end{eqnarray*}
  We also introduce four types of three-dimensional cones in $N_6$:
  \begin{eqnarray*}
    &&\text{BBB}\quad
       \la \frac{e_1+e_2}{2}, \frac{e_2+e_3}{2}, \frac{e_3+e_1}{2} \ra\\
    &&\text{BBC}\quad
       \la e_1, \frac{e_1+e_2}{2}, \frac{e_1+e_3}{2} \ra\\
    &&\text{BCD} \quad
       \la e_1, \frac{e_1+e_2}{2}, \frac{e_1-e_3}{2} \ra\\
    &&\text{BDD}\quad
       \la \frac{e_1+e_2}{2}, \frac{e_1-e_3}{2}, \frac{e_2-e_3}{2} \ra\\
    &&\text{CDD}\quad
       \la e_1, \frac{e_1-e_2}{2}, \frac{e_1-e_3}{2} \ra\\
  \end{eqnarray*}
  Also, let $\Gamma_5 \subset \Gamma_6$ be the subgroup
  \begin{displaymath}
    \Gamma_5 := \Stab_{\Gamma_6} \{ (R_1,G_1,B_1) \}
    \simeq C_3\times S_2 = C_3 \times \la\Cr\ra.
  \end{displaymath}
  It acts on $N_4$ (the sublattice of $\bZ^4$ with even sums of
  coordinates) by permuting the last three coordinates and sending $v
  \to -v$. It is also the stabilizer of $N_5$ in $\Gamma_6$.
\end{notations}

\begin{lemma}\label{lem:F5}
  The fan $\fF_5$ is a nonsingular fan with $18$ rays, comprising the
  $\Gamma_5$-orbits of the B, C, D rays defined above ($6$ of each
  kind), and with $32$ three-dimensional cones comprising the
  $\Gamma_5$-orbits of the cones defined above: BBB($2$), BBC($6$),
  BCD($12$), BDD($6$), CDD($6$).  Thus, $Z_5\tor$ is a nonsingular toric
  variety.
\end{lemma}

The first blowup $\rho_1{}_*\inv Z_5\tor$ is the blowup of this toric
variety $Z_5\tor$ at the origin. 
Of the six curves of Equation~\eqref{eq:six-curves}, three lie
in $T_5$:
\begin{equation}\label{eq:3F-curves}
  \{(1,x\inv,1,1),\  (1,1,x\inv,1),\  (1,1,1,x\inv), \quad x\in\bC^* \}
\end{equation}
and the other three intersect $T_5$ transversally. Thus, the second
blowup $\rho_2\colon Z_5 \to \rho_1{}_*\inv Z_5\tor$ is a blowup at
three disjoint $\bP^1$'s producing three type F divisors.

\smallskip\emph{Step 2.} We find all the possible fibers of the
restricted family $\cY_6|_{Z_5} \to Z_5$ by examining the surfaces of
Figure~\ref{fig-deg6} and picking those where it is possible to have
the point $P_{111} = R_1\cap G_1\cap B_1$.  which must be smooth by
Lemma~\ref{lem:n-pts}. These possibilities are listed in
Figure~\ref{fig-deg5}. We omit the generic case of $Y_6=\Sigma$. We
do not list the B, G, H divisors which do not change the underlying
surface.

Let $\cY'_5$ be the blowup of $\cY_6|_{Z_5}$ along the section
corresponding to the points $P_{111}$. From the computations of
Section~\ref{sec:evolution}, we conclude that $\cL'$, the double
of the relative log canonical divisor, is nef on all fibers. It is
also big on all fibers except those lying in the three type F
divisors. In that case, there is an irreducible component of the fiber
isomorphic to the blowup of $\hash8(1)\, [\bP^2]$ which has volume
$1-1=0$. Obviously $\Bl_1\bP^2 = \bF_1$. The contraction
$\cY'_5\to \cY_5\can$ given by the linear system $|m\cL'|$ for $m\gg0$
contracts this $\bF_1$ to $\bP^1$ along the $\bP^1$-fibration.

Over an open subset of $Z_5$ the fibers are del Pezzo surfaces of
degree $5$, which we denote by $0_1(5)$. 
All other irreducible components of the stable surfaces $Y_5$ have already
appeared before, in degrees $3$ and $4$.

The fibers, labeled pairs, in the family $\cY_5\can\to Z_5$ are all
distinct with one exception: for points $s\in \bP^1\subset Z_5$ with the same
image $\rho_2(s) \in \rho_1{}_*\inv Z_5\tor$ the fibers are
isomorphic. Let $Z_5 \to \oZ_5 := \rho_1{}_*\inv Z_5\tor$ be the
undoing of the second blowup. The family $\cY_5\can\to Z_5$ descends
to a family $\overline{\cY}_5\can\to \oZ_5$ in which all fibers are
distinct.

Figure~\ref{fig-final5} shows the KSBA pairs appearing in the above families,
and Table~\ref{tab:deg5} lists their irreducible components.

\smallskip\emph{Step 3.} In this case, the relabeling group is the
group $\Gamma_5 = C_3\times S_2$ defined in Notations~\ref{not:F5}. It
acts faithfully on $T_5$ and $\oZ_5$.

Arguing as in the previous sections, we obtain a $\bZ_2^2$-gerbe over
the quotient stack $[\oZ_5 : \Gamma_5]$ whose coarse moduli space
$Z_5 / \Gamma_5$ has a finite map to the closure of $M\ubur_5$ in
$\oM\slc$. Thus, $\oM\ubur_5 = \oZ_5/\Gamma_5$.  Finally, we enumerate
the boundary divisors. There is one each of types B, C, D, G, E, and
two of type H. We do not consider the F divisor to be part of the
boundary, since the surfaces there are canonical. This completes the
proof of Theorem~\ref{thm-intro:5}.

\bibliographystyle{amsalpha} 

\begin{thebibliography}{Bur66}

\bibitem[Ale25]{alexeev2025kappa-classes}
Valery Alexeev, \emph{Kappa classes on {KSBA} spaces}, Moduli \textbf{2}
  (2025), arXiv:2309.14842.

\bibitem[AP12]{alexeev2012non-normal-abelian}
Valery Alexeev and Rita Pardini, \emph{Non-normal abelian covers}, Compos.
  Math. \textbf{148} (2012), no.~4, 1051--1084. \MR{2956036}

\bibitem[AP24]{alexeev2024explicit-compactifications}
\bysame, \emph{Explicit compactifications of moduli spaces of {C}ampedelli and
  {B}urniat surfaces}, Annali della Scuola Normale Superiore di Pisa. Classe
  di Scienze \textbf{40} (2024), arXiv:0901.4431.

\bibitem[BC10]{bauer2010burniat-II}
I.~Bauer and F.~Catanese, \emph{Burniat surfaces. {II}. {S}econdary {B}urniat
  surfaces form three connected components of the moduli space}, Invent. Math.
  \textbf{180} (2010), no.~3, 559--588. \MR{2609250}

\bibitem[BC13]{bauer2013burniat-III}
\bysame, \emph{Burniat surfaces {III}: deformations of automorphisms and
  extended {B}urniat surfaces}, Doc. Math. \textbf{18} (2013), 1089--1136.
  \MR{3138841}

\bibitem[BC14]{bauer2014burniat-II-erratum}
\bysame, \emph{Erratum to: {B}urniat surfaces {II}: secondary {B}urniat
  surfaces form three connected components of the moduli space}, Invent. Math.
  \textbf{197} (2014), no.~1, 237--240. \MR{3219518}

\bibitem[Bur66]{burniat1966sur-les-surfaces}
Pol Burniat, \emph{Sur les surfaces de genre {$P\sb{12}>1$}}, Ann. Mat. Pura
  Appl. (4) \textbf{71} (1966), 1--24. \MR{206810}

\bibitem[CC13]{chan2013kulikov-surfaces}
Tsz On~Mario Chan and Stephen Coughlan, \emph{Kulikov surfaces form a connected
  component of the moduli space}, Nagoya Math. J. \textbf{210} (2013), 1--27.
  \MR{3079273}

\bibitem[Hu14]{hu14compactifications-of-moduli}
Xiaoyan Hu, \emph{The compactifications of moduli spaces of {B}urniat surfaces
  with {$2\leq K^{2} \leq5$}}, Ph.D. thesis, University of Georgia, 2014.

\bibitem[Ino94]{inoue1994some-new-surfaces}
Masahisa Inoue, \emph{Some new surfaces of general type}, Tokyo J. Math.
  \textbf{17} (1994), no.~2, 295--319. \MR{MR1305801 (95j:14048)}

\bibitem[Kol23]{kollar2023families-of-varieties}
J\'{a}nos Koll\'{a}r, \emph{Families of varieties of general type}, Cambridge
  Tracts in Mathematics, vol. 231, Cambridge University Press, Cambridge, 2023.
  \MR{4566297}

\bibitem[Kul04]{kulikov2004old-examples}
Vik.\~S. Kulikov, \emph{Old examples and a new example of surfaces of general
  type with {$p_g=0$}}, Izv. Ross. Akad. Nauk Ser. Mat. \textbf{68} (2004),
  no.~5, 123--170. \MR{2104852}

\bibitem[Par91]{pardini1991abelian-covers}
Rita Pardini, \emph{{A}belian covers of algebraic varieties}, J. Reine Angew.
  Math. \textbf{417} (1991), 191--213. \MR{1103912 (92g:14012)}

\bibitem[Pet77]{peters1977on-certain-examples}
C.~A.~M. Peters, \emph{On certain examples of surfaces with {$p_{g}=0$} due to
  {B}urniat}, Nagoya Math. J. \textbf{66} (1977), 109--119. \MR{0444676 (56
  \#3026)}

\bibitem[{Sag}22]{sagemath}
{Sage Developers}, \emph{{S}agemath, the {S}age {M}athematics {S}oftware
  {S}ystem ({V}ersion 9.5)}, 2022, {\tt https://www.sagemath.org}.

\end{thebibliography}

\def\cprime{$'$} 
\providecommand{\bysame}{\leavevmode\hbox to3em{\hrulefill}\thinspace}
\providecommand{\MR}{\relax\ifhmode\unskip\space\fi MR }
\providecommand{\MRhref}[2]{%
  \href{http://www.ams.org/mathscinet-getitem?mr=#1}{#2}
}
\providecommand{\href}[2]{#2}

\newpage

\appendix

\section{Figures}
\label{sec:figures}


${}$

\begin{figure}[H]
  \centering
  \includegraphics[width=364pt]{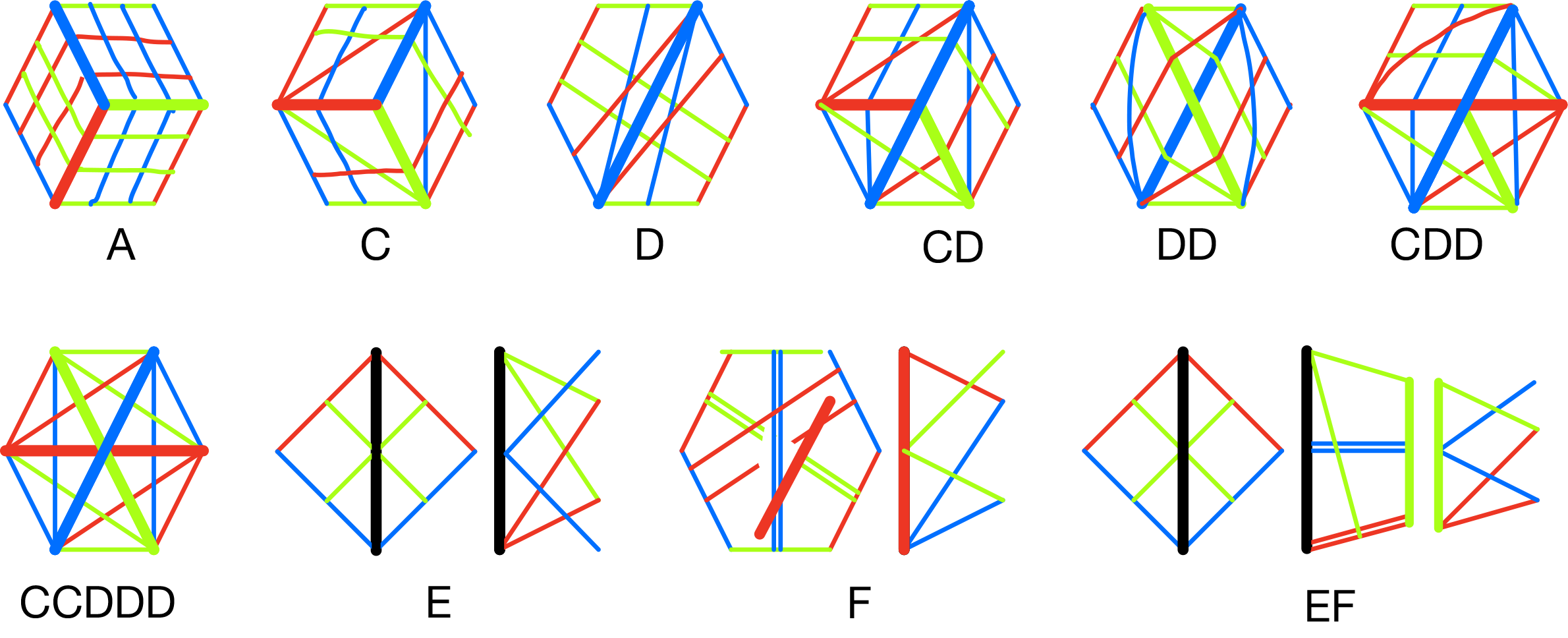} 
  \caption{Surfaces underlying the stable pairs \\$(Y_6,\sum_{i=0}^3\frac12(R_i+G_i+B_i))$ in $\oM_6$} 
  \label{fig-deg6}
\end{figure}

\begin{figure}[b]
  \centering 
  \includegraphics[width=254pt]{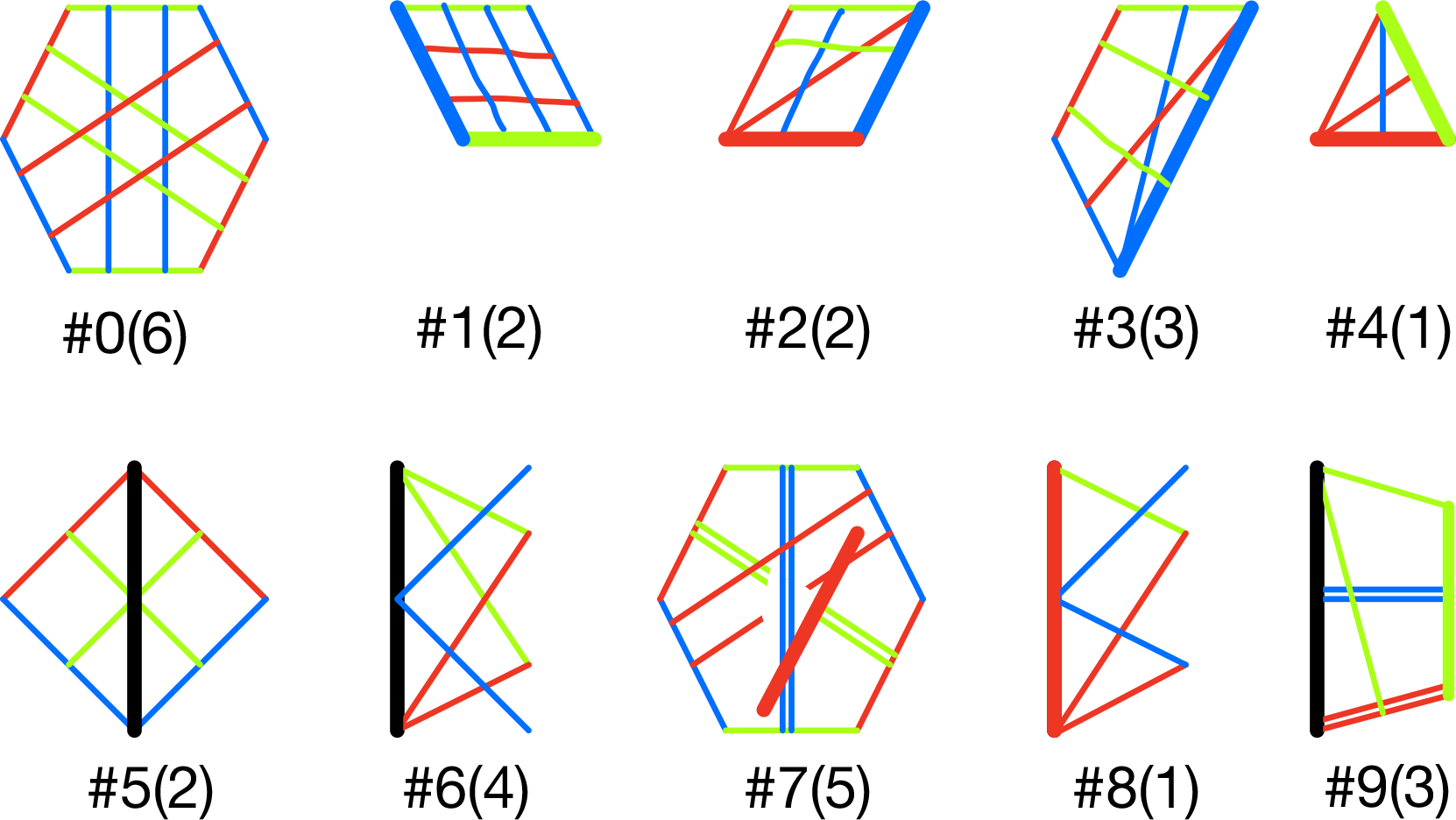} 
  \caption{Irreducible components of $Y_6$ appearing in $\oM_6$, with
    volumes. The underlying surfaces are: \\
    \protect\begin{displaymath}
      \protect\kern -70pt
      \protect\begin{array}[]{lllll} 
                \hash0(6)\ \Sigma 
                &\hash1(2)\,\bP^1\times\bP^1 
                &\hash2(2)\,\bP^1\times\bP^1 &\hash3(3)\, \bF_1
                &\hash4(1)\,\bP^2\\ 
                \hash5(2)\,\bP^1\times\bP^1 &\hash6(4)\,\bP^2 
                &\hash7(5)\,\Bl_1\Sigma &\hash8(1)\,\bP^2 &\hash9(3)\,\bF_1\\ 
              \protect\end{array} 
            \protect\end{displaymath}
        }
  \label{fig-irr6}
\end{figure}

\FloatBarrier
\begin{figure}[H]
  \centering 
  \includegraphics[width=346pt]{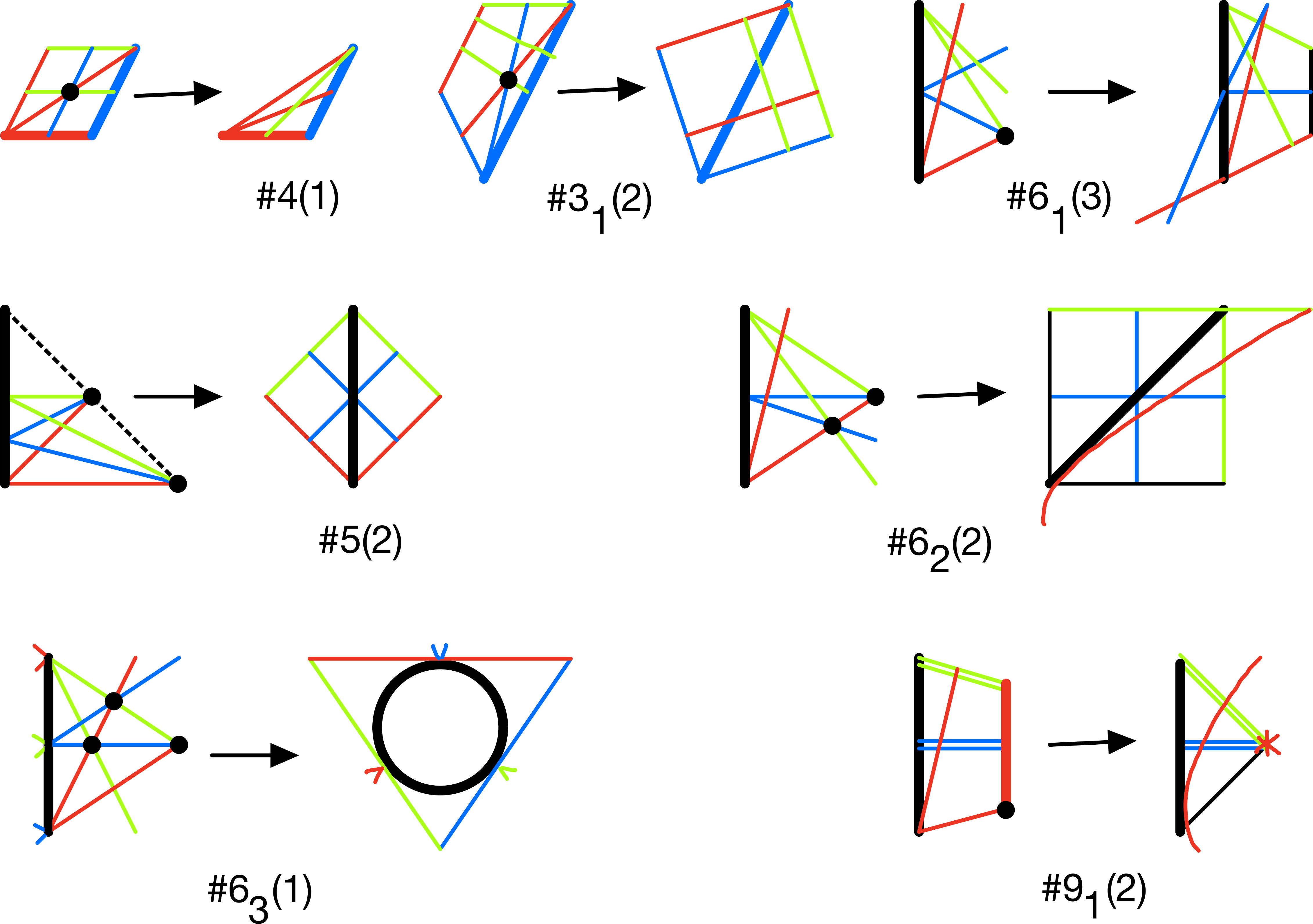}
  \caption{New irreducible components appearing after blowups and
    contractions by $|mL'|$. The underlying surfaces are:\\
    \protect\begin{displaymath}
      \protect\kern -70pt
      \protect\begin{array}[]{lllll}
        \hash3_1(2)\ \bP^1\times\bP^1
        &\hash6_1(3)\ \bF_1
        &\hash6_2(2)\ \bP^1\times\bP^1
        &\hash6_3(1)\ \bP^2
        &\hash9_1(2)\ \bP(1,1,2)
              \protect\end{array}
          \protect\end{displaymath}
      }
  \label{fig-new-irr}
\end{figure}

\FloatBarrier\begin{figure}[H]
  \centering
    \includegraphics[width=233pt]{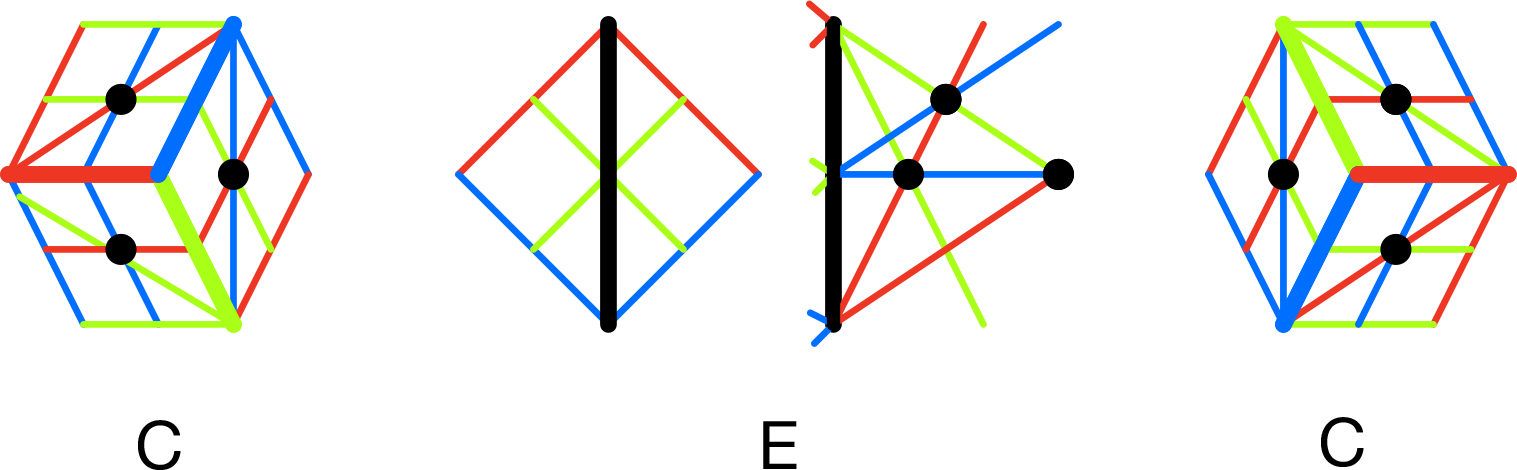}
    \caption{Surfaces in the family $\cY_6|_{Z_3} \to Z_3$}
    \label{fig-deg3}
\end{figure}

\FloatBarrier\begin{figure}[H]
  \centering
  \includegraphics[width=200pt]{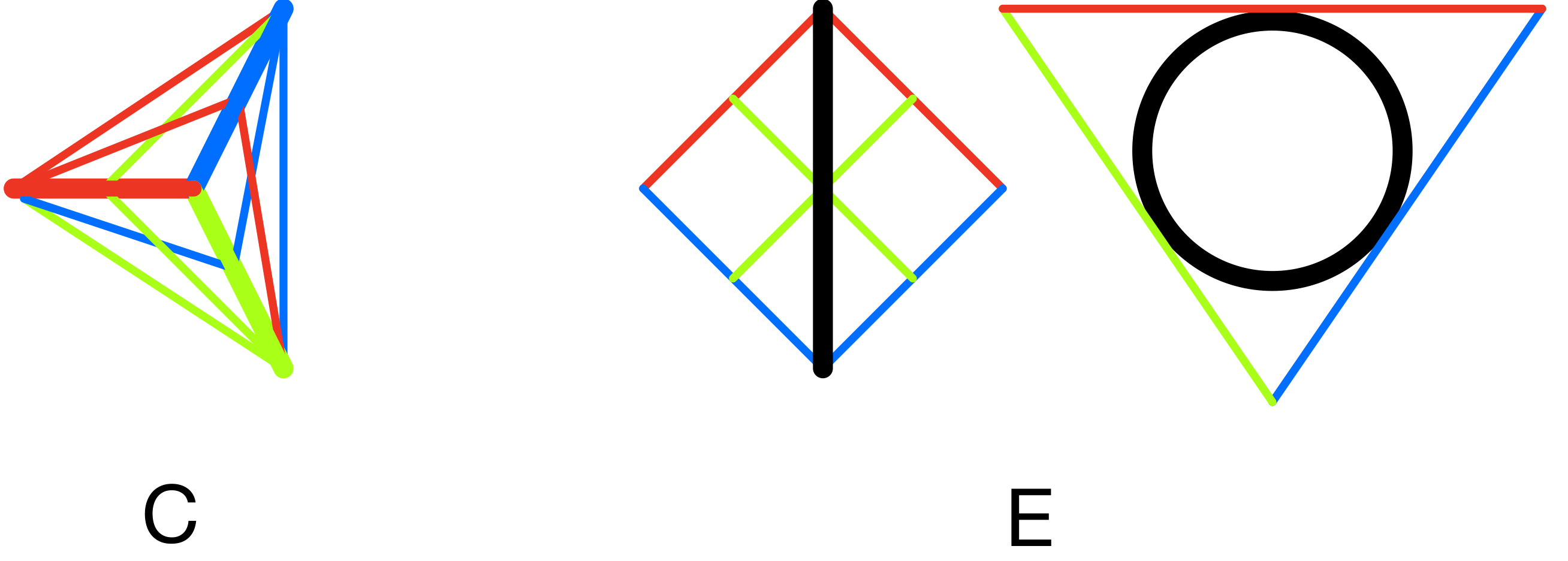} 
  \caption{KSBA stable surfaces appearing in $\oM_3$}
  \label{fig-final3}    
\end{figure}

\FloatBarrier\begin{figure}[H]
  \centering
  \includegraphics[width=221pt]{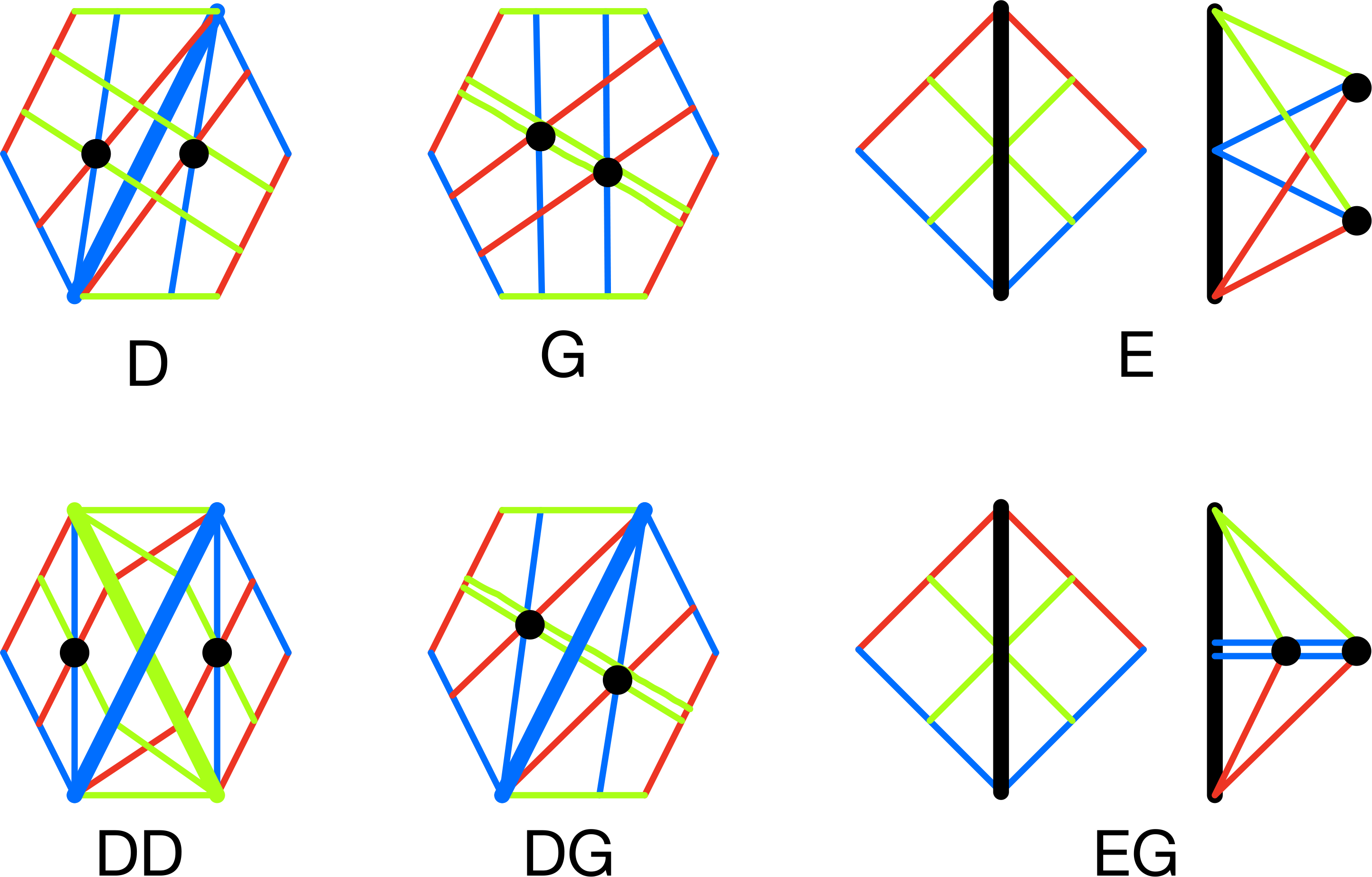}  
  \caption{Surfaces in the family $\cY_6|_{Z\foura} \to Z\foura$}
  \label{fig-deg4a}
\end{figure}

\FloatBarrier\begin{figure}[H]
  \centering 
  \includegraphics[width=290pt]{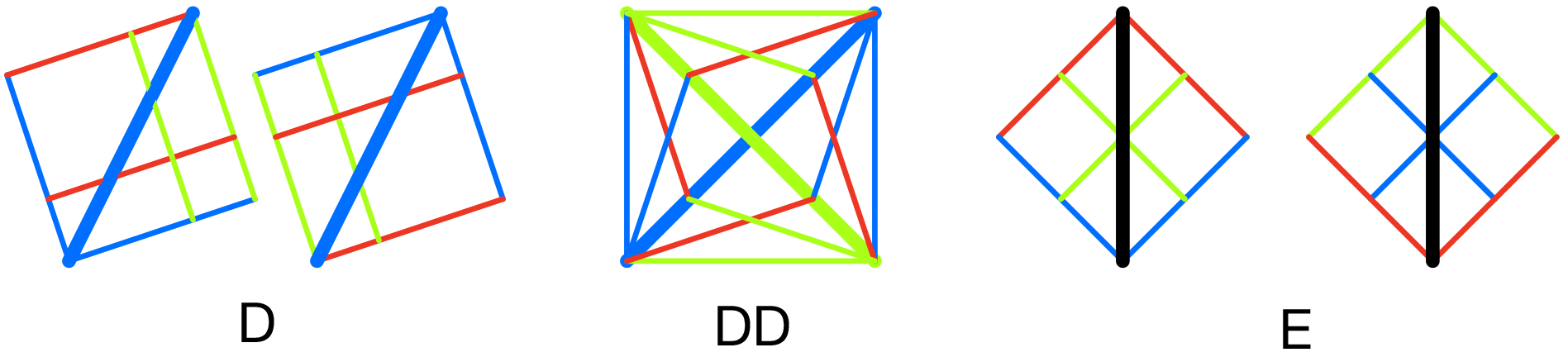}   
  \caption{KSBA stable surfaces appearing in $\oM\foura$}
  \label{fig-final4a}
\end{figure}

\FloatBarrier\begin{figure}[H]
  \centering
  \includegraphics[width=301pt]{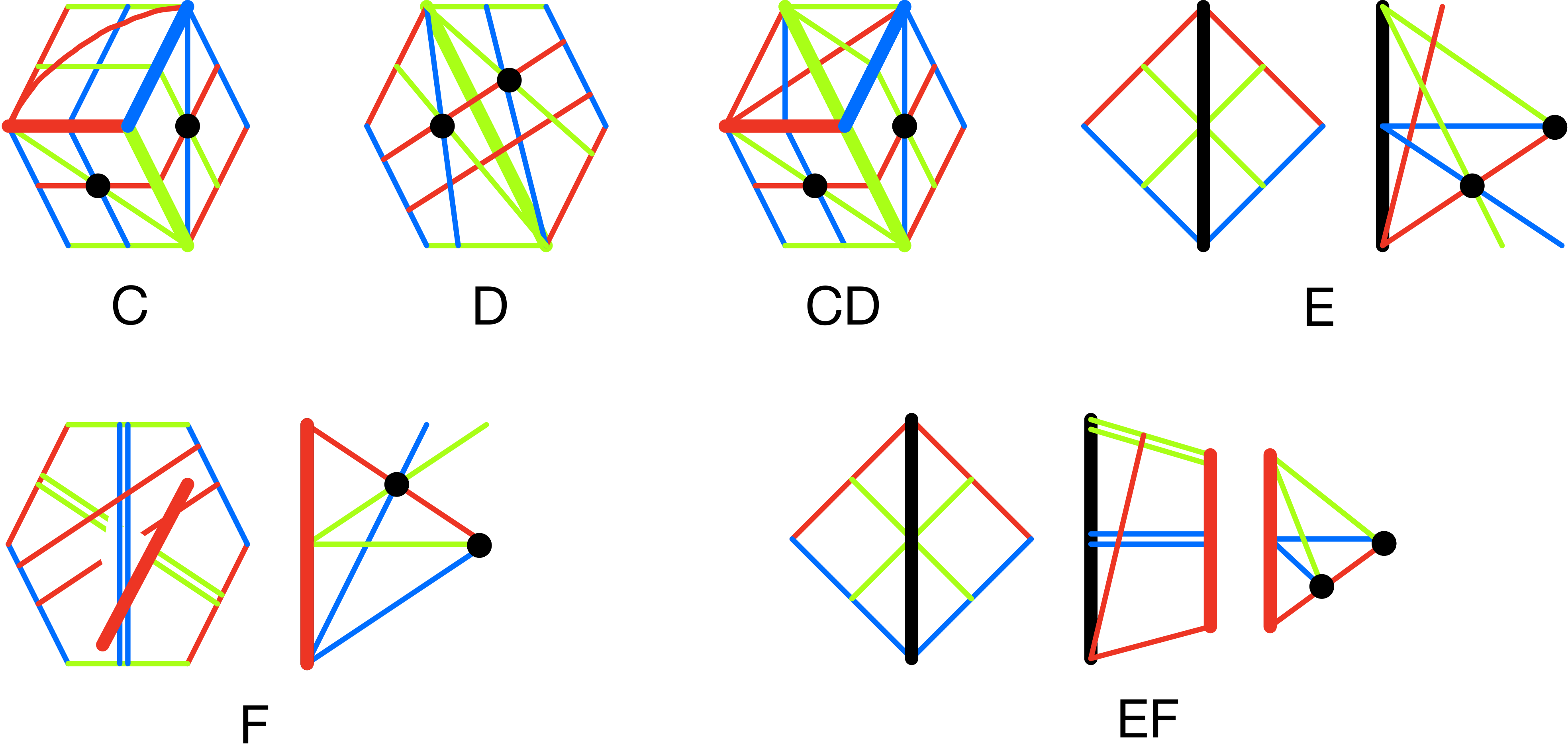}  
  \caption{Surfaces in the family $\cY_6|_{Z\fourb} \to Z\fourb$}
  \label{fig-deg4b}
\end{figure}
\FloatBarrier\begin{figure}[H]
  \centering 
  \includegraphics[width=244pt]{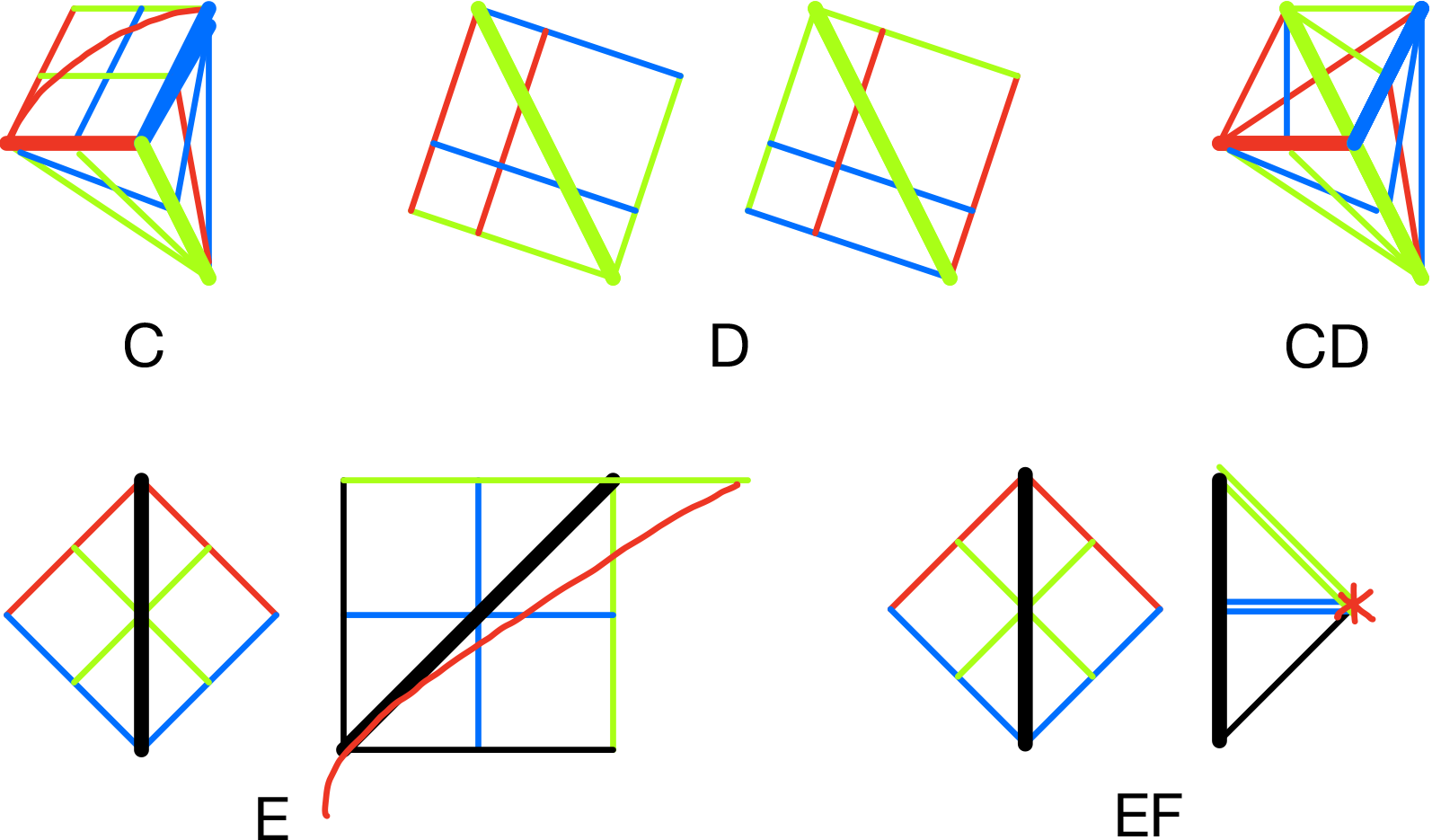} 
  \caption{KSBA stable surfaces appearing in $\oM\fourb$}
  \label{fig-final4b}
\end{figure}

\FloatBarrier\begin{figure}[H]
  \centering
  \includegraphics[width=357pt]{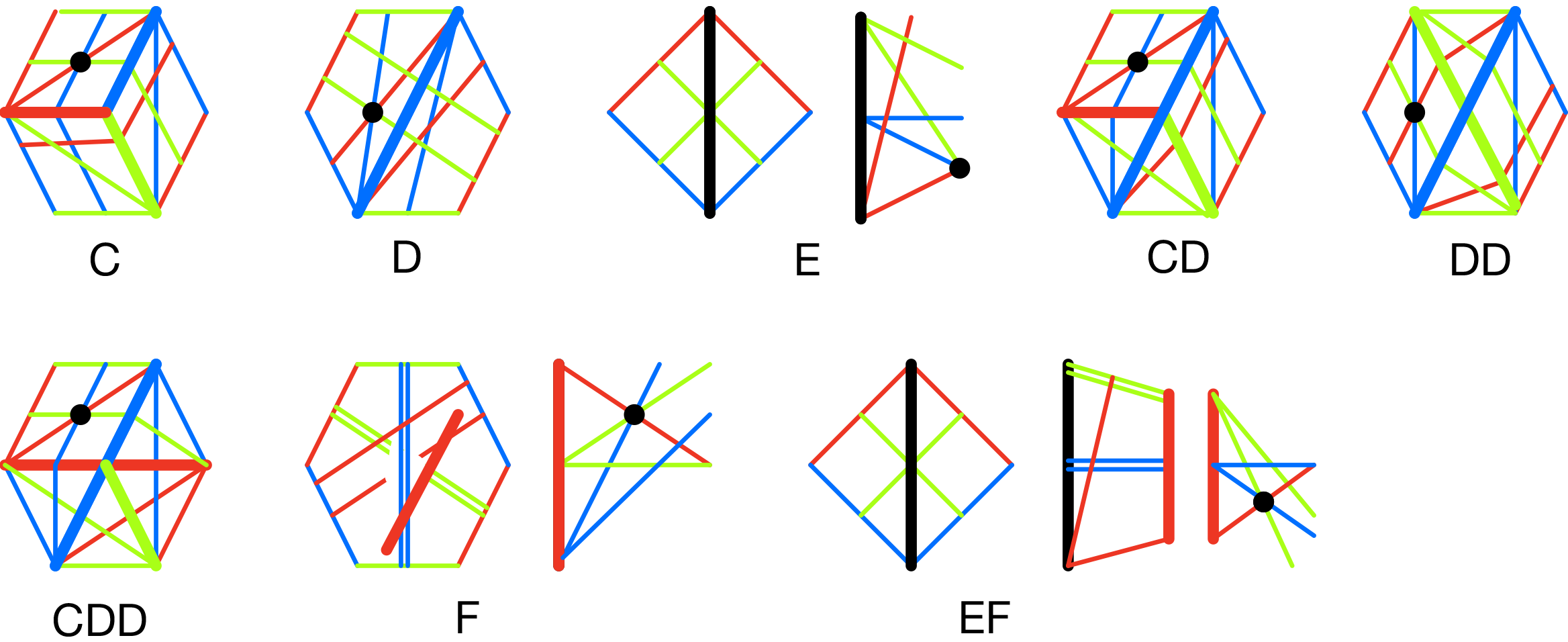}  
  \caption{Surfaces in the family $\cY_6|_{Z_5} \to Z_5$}
  \label{fig-deg5}
\end{figure}

\FloatBarrier\begin{figure}[H]
  \centering 
  \includegraphics[width=280pt]{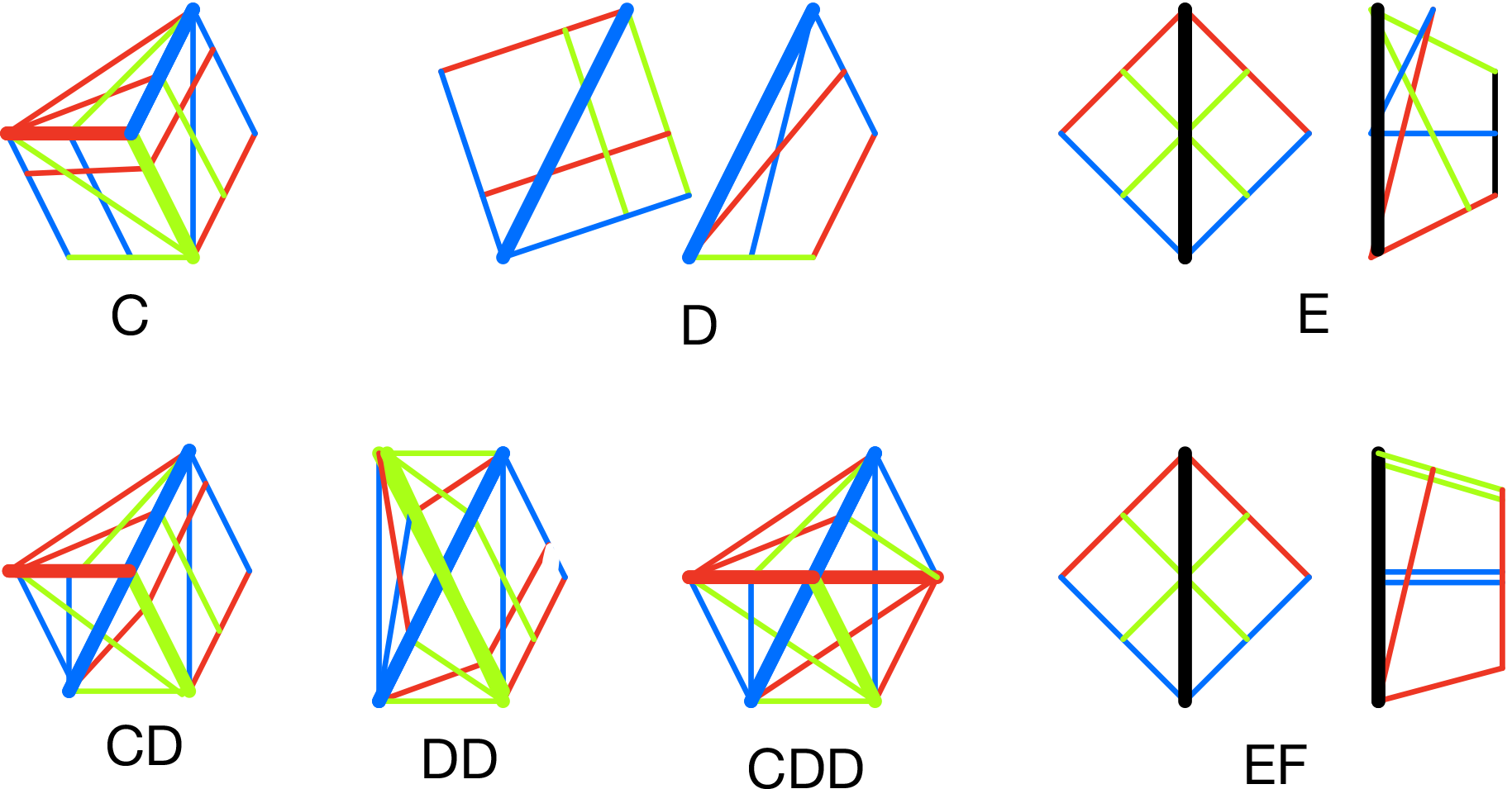}  
  \caption{KSBA stable surfaces appearing in $\oM_5$}
  \label{fig-final5}
\end{figure}

\FloatBarrier

\section{Tables}
\label{sec:tables}

\begin{table}[H] 
  \centering
  \begin{tabular}{ll}
    case &components with volumes\\
    \hline
    0 &$\hash 0(6)$\\
    A &$\hash 1(2)$, $\hash 1(2)$, $\hash 1(2)$\\
    C &$\hash 2(2)$, $\hash 2(2)$, $\hash 2(2)$\\
    D &$\hash 3(3)$, $\hash 3(3)$\\
    CD &$\hash 2(2)$, $\hash2(2)$, $\hash4(1)$, $\hash4(1)$\\
    DD &$\hash2(2)$, $\hash2(2)$, $\hash4(1)$, $\hash4(1)$\\
    CDD &$\hash2(2)$, $\hash4(1)$, $\hash4(1)$, $\hash4(1)$, $\hash4(1)$\\
    CCDDD &$\hash4(1))$, $\hash4(1)$, $\hash4(1)$, $\hash4(1)$,
            $\hash4(1)$, $\hash4(1)$\\
    E &$\hash 5(2)$, $\hash 6(4)$\\
    F &$\hash7(5)$, $\hash8(1)$\\ 
    EF &$\hash5(2)$, $\hash9(3)$, $\hash8(1)$\\
  \end{tabular}
  \caption{Surfaces underlying stable pairs in $\oM_6$, with volumes}
  \label{tab:deg6}
\end{table}

\begin{table}[H] 
  \centering
  \begin{tabular}{ll}
    case &components with volumes\\
    \hline
    0 &$\hash0_3(3)$\\
    C &$\hash4(1)$, $\hash4(1)$, $\hash4(1)$\\
    E &$\hash5(2)$, $\hash6_3(1)$
  \end{tabular}
  \caption{Surfaces underlying stable pairs in $\oM_{3}$}
  \label{tab:deg3}
\end{table}

\begin{table}[H] 
  \centering
  \begin{tabular}{ll}
    case& components with volumes\\
    \hline
    0, G &$\hash0_2(4)$\\
    D, DG &$\hash 3_1(2)$, $\hash 3_1(2)$\\
    E, EG &$\hash5(2)$, $\hash5(2)$\\
    DD &$\hash4(1)$, $\hash4(1)$, $\hash4(1)$, $\hash4(1)$
  \end{tabular}
  \caption{Surfaces underlying stable pairs in $\oM\foura$}
  \label{tab:deg4a} 
\end{table}

\begin{table}[H] 
  \centering
  \begin{tabular}{lllll}
    case &components with volumes\\
    \hline
    0, F &$0_2(4)$\\
    C &$\hash4(1)$, $\hash4(1)$, $\hash2(2)$\\
    D &$\hash3_1(2)$, $\hash3_1(2)$\\
    CD &$\hash4(1)$, $\hash4(1)$, $\hash4(1)$, $\hash4(1)$\\
    E &$\hash5(2)$, $\hash6_2(2)$\\
    EF &$\hash5(2)$, $\hash9_2(2)$
  \end{tabular}
  \caption{Surfaces underlying stable pairs in $\oM\fourb$}
  \label{tab:deg4b}
\end{table}

\begin{table}[H] 
  \centering
  \begin{tabular}{ll}
    case &components with volumes\\
    \hline
    0 &$0_1(5)$\\
    C &$\hash 4(1)$, $\hash 2(2)$, $\hash 2(2)$\\
    D &$\hash 3_1(2)$, $\hash 3(3)$\\
    E &$\hash 5(2)$, $\hash 6_1(3)$\\
    CD &$\hash4(1)(1)$, $\hash2(2)$, $\hash4(1)$, $\hash4(1)$\\
    DD &$\hash4(1)$, $\hash2(2)$, $\hash4(1)$, $\hash4(1)$\\
    CDD &$\hash4(1)$, $\hash4(1)$, $\hash4(1)$, $\hash4(1)$, $\hash4(1)$\\
    F &$\hash7(5)$\\
    EF &$\hash5(2)$, $\hash9(3)$
  \end{tabular}
  \caption{Surfaces underlying stable pairs in $\oM_5$}
  \label{tab:deg5}
\end{table}

\end{document}